\documentclass{amsart}

\usepackage[utf8x]{inputenc}
\usepackage{microtype}
\usepackage{lmodern}
\usepackage[T1]{fontenc}
\usepackage{ucs}
\usepackage[english]{babel}

\usepackage{amsmath,amsfonts,amssymb,amsthm,amsxtra}
\usepackage{stmaryrd}

\usepackage{tikz,tikz-3dplot,tikz-cd}

\usepackage[small]{caption}
\usepackage{csquotes}
\usepackage{fullpage}
\usepackage[colorlinks,plainpages]{hyperref}

\usepackage{paralist}
\usepackage{enumitem}
\usepackage{todonotes}

\theoremstyle{plain}
\newtheorem*{introtheorem}{Theorem}

\newtheorem{theorem}{Theorem}[section]
\newtheorem{proposition}[theorem]{Proposition}
\newtheorem{lemma}[theorem]{Lemma}

\theoremstyle{definition}
\newtheorem{definition}[theorem]{Definition}
\newtheorem{example}[theorem]{Example}

\theoremstyle{remark}
\newtheorem{remark}[theorem]{Remark}

\newlist{deflist}{enumerate}{2}
\setlist[deflist]{label=(\arabic*),ref=(\arabic*)}

\newlist{statements}{enumerate}{2}
\setlist[statements]{label=(\arabic*),ref=(\arabic*)}

\newlist{equivalence}{enumerate}{2}
\setlist[equivalence]{label=(\alph*),ref=(\alph*)}

\newcommand{\abr}[1]{{\left\langle#1\right\rangle}} 
\newcommand{\cbr}[1]{{\left\{#1\right\}}} 
\newcommand{\br}[1]{{\left(#1\right)}} 
\newcommand{\sbr}[1]{{\left[#1\right]}} 
\newcommand{\ssbr}[1]{{\left[\left[#1\right]\right]}} 
\newcommand{\vbr}[1]{{\left\lvert#1\right\rvert}} 

\DeclareMathOperator{\Hom}{Hom}
\DeclareMathOperator{\End}{End}
\DeclareMathOperator{\ord}{ord}

\DeclareMathOperator{\pr}{pr}
\DeclareMathOperator{\Max}{Max}
\DeclareMathOperator{\Min}{Min}

\newcommand{\KK}{k}
\newcommand{\CC}{\mathbb{C}}
\newcommand{\NN}{\mathbb{N}}
\newcommand{\ZZ}{\mathbb{Z}}

\newcommand{\ee}{\mathbf{e}}
\newcommand{\zero}{\mathbf{0}}
\newcommand{\one}{\mathbf{1}}

\newcommand{\ol}{\overline}
\newcommand{\wh}{\widehat}

\newcommand{\reg}{\mathrm{reg}}

\newcommand{\RFI}{\mathcal{R}} 
\newcommand{\VR}{\mathcal{V}} 
\newcommand{\GSI}{\mathcal{G}} 

\newcommand{\fracid}{\mathfrak} 

\title{Gorenstein Endomorphism Rings on Curve Singularities}
\author{Philipp Korell}
\address{Philipp Korell\\
adesso insurance solutions GmbH}
\email{\href{mailto:philipp.korell@adesso-insurance-solutions.de}{philipp.korell@adesso-insurance-solutions.de}}



\begin{document}


\begin{abstract}
We characterize the Gorensteinness of endomorphism rings of a fractional ideal on a curve singularity by properties of the ideal. Moreover, the Gorenstein algebroid curves with only Gorenstein integral extensions are classified.
\end{abstract}


\maketitle


\section{Introduction}


By the Cayley--Hamilton theorem, endomorphism rings of modules are integral extensions of the base ring.  Consequently, endomorphism rings appear in the contexts of (algorithmic) normalization or non-commutative resolutions \cite{leuschke12}. Grauert and Remmert showed that a reduced Noetherian ring is normal if and only if it is the endomorphism ring of a so-called \emph{test ideal} \cite{grauert1971}. This criterion for normality is the foundation of de Jong's algorithm for normalization \cite{j.symb.comp.26.273,decker99}: Computing successively the endomorphism rings of test ideals yields a sequence of integral extensions, eventually stabilizing at the normalization. For one-dimensional rings, Leuschke related the sequence of endomorphism rings obtained by this procedure to non-commutative crepant resolutions \cite{can.j.math.59.332}.

In dimension one, normalization is the same as desingularization. The unique test ideal of a local algebroid curve is the maximal ideal. In this setting, B\"ohm, Decker, and Schulze studied the intermediate rings produced in the algorithmic normalization of algebroid curves according to de Jong \cite{int.j.alg.comp.24.69}. As a main tool for computing the endomorphism ring of the maximal ideal they considered the \emph{semigroup of values} of an algebroid curve.

In this article, we consider endomorphism rings of fractional ideals of a certain class of one-dimensional Cohen--Macaulay rings which we call \emph{admissible} (see Definition~\ref{030}). These admissible rings generalize algebroid curves while still allowing for the construction of a semigroup of values.

For an admissible ring $R$, there exist only finitely many (discrete) valuation rings of the total ring of fractions $Q_R$ which contain $R$. Considering simultaneously the corresponding discrete valuations $\nu = \br{\nu_1, \ldots, \mu_s} \colon Q_R \to \br{\ZZ \cup \infty}^s$, we define the semigroup of values as the set of values of non-zero divisors of $R$: $\Gamma_R = \nu \br{R^\reg} \subset \NN^s$. Similarly, to any fractional ideal $\fracid I$ of $R$ we may associate a \emph{value semigroup ideal} $\Gamma_{\fracid I}$.

It is a particularly useful feature of value semigroups that a certain symmetry condition determines whether an admissible ring is Gorenstein. This symmetry was first discovered by Kunz for irreducible curves \cite{proc.am.math.soc.25.748} and later extended by Delgado to reducible curves \cite{man.math.61.285}. From Kunz' symmetry condition, J\"ager derived a characterization of canonical ideal in terms of their value semigroup ideals \cite{arch.math.29.504}. Following Delgado's notion, D'Anna extended this characterization to reduced (but not necessarily irreducible) admissible rings \cite{comm.algebra.25.2939}.

Barucci, D'Anna, and Fr\"oberg took a closer look at the semiring structure of value semigroups \cite{j.pure.app.alg.147.215}. They introduced the class of \emph{good semigroups} which are submonoids of $\NN^s$ satisfying certain properties of value semigroups first identified by Delgado \cite{man.math.61.285}. Notably, they also gave an example of a good semigroup which is not the semigroup of values of an admissible ring.

Nonetheless, good semigroups together with the so-called \emph{good semigroup ideals} can be regarded as \enquote{combinatorial versions} of admissible rings and their fractional ideals. In this spirit, in \cite{j.commut.algebra.11.81} it was shown that every good semigroup possesses a \emph{canonical semigroup ideal} which induces a purely combinatorial version of the Cohen--Macaulay duality on fractional ideals. In fact, this canonical semigroup ideal is characterized by its dualizing property. This is in particular not obvious as the set of good semigroup ideals is in general not closed under taking differences (which is the corresponding operation to ideal quotients). Moreover, this \enquote{combinatorial duality} is compatible with the algebraic one under taking values.

We apply these combinatorial methods to determine the Gorensteinness of the endomorphism ring $\End \br{\fracid I}$ of a fractional ideal $\fracid I$ of an admissible ring $R$. The case of $\fracid I$ being the maximal ideal $\fracid m_R$ of a local admissible ring $R$ is solved already: Barucci and Fr\"oberg stated that $\End \br{\fracid m_R}$ is Gorenstein if and only if $R$ is almost Gorenstein of maximal embedding dimension \cite{j.alg.188.418}. Considering also analytically ramified rings, Goto, Matsuoka, and Phuong provided a full proof of this statement \cite{j.alg.379.355}.

As shown by Lipman, $R$ has maximal embedding dimension if and only if $\fracid m_R$ is stable \cite{amer.j.math.93.649}. In fact, we show that for any fractional ideal $\fracid I$, the Gorensteinness of $\End \br{\fracid I}$ implies stability of $\fracid I$. The proof makes use of the corresponding statement for value semigroups. This leads to the following characterization:

\begin{introtheorem}[Theorem~\ref{001}]
Let $R$ be an admissible ring, and let $\fracid I$ be a fractional ideal of $R$. Then the following are equivalent:
\begin{equivalence}
\item
The endomorphism ring $\End \br{\fracid I}$ is Gorenstein.
\item
The fractional ideal $\fracid I$ is stable and a canonical ideal of $\End \br{\fracid I}$.
\item
The fractional ideal $\fracid I$ is stable and self-dual.
\end{equivalence}
\end{introtheorem}

For the second main result of this paper we turn to the case of $\fracid I = \fracid m_R$ being the maximal ideal of a local admissible ring $R$. Then we classify the simultaneous Gorensteinness of $R$ and $\End \br{\fracid m_R}$:

\begin{introtheorem}[Theorem~\ref{026}]
Let $R$ be a local admissible ring with maximal ideal $\fracid{m}_R$. The following are equivalent.
\begin{equivalence}
\item
Every integral extension of $R$ in its total ring of fractions $Q_R$ is Gorenstein.
\item
$R$ and $\fracid{m}_R : \fracid{m}_R$ are Gorenstein.
\item
$R$ is Gorenstein and has maximal embedding dimension.
\item
$\Gamma_R$ and $\Gamma_R \setminus \cbr{\zero} - \Gamma_R \setminus \cbr{\zero}$ are symmetric.
\end{equivalence}

Moreover, if $R$ contains a field, and if $\KK = R / \fracid{m}_R$ is algebraically closed, then the above statements are equivalent to $\wh R \cong \KK \ssbr{x,y} / \abr{x^{n+1} - y^2}$ for some $n \in \NN$.
\end{introtheorem}

Sections \ref{108} -- \ref{110} provide the necessary foundations on admissible rings, value semigroups, resp.\ good semigroups, and the symmetry condition. Section \ref{111} is concerned with stating and proving Theorem \ref{001}. In Section \ref{112}, we establish a combinatorial version of Theorem \ref{026} which we then use to prove Theorem \ref{026} in Section \ref{113}.


\subsection*{Notations}


By a \emph{ring} $R$ we mean a commutative unitary ring. The set of maximal ideals of $R$ is denoted by $\Max \br{R}$, and the set of minimal prime ideals by $\Min \br{R}$. We write $R^\reg$ for the set of non-zero divisors (i.e.\ the \emph{regular} elements) of $R$. The \emph{total ring of fractions} of $R$ is the localization $Q_R = \br{R^\reg}^{-1} R$. For a subset $\fracid I$ of $Q_R$ we use the notation $\fracid I^\reg = \fracid I \cap Q_R^\reg$. The set of units of $R$ is denoted by $R^\ast$. In particular, $Q_R^\ast = Q_R^\reg$.

The integral closure of $R$ in $Q_R$ is denoted by $\ol R$, also called the \emph{normalization} of $R$ if $R$ is reduced. Finally, if $R$ is semilocal, then $\wh R$ is the \emph{completion} of $R$ at the \emph{Jacobson radical} $\fracid j_R = \bigcap_{\fracid m \in \Max \br{R}} \fracid m$.


\section{Preliminaries}
\label{108}


A \emph{complex algebroid curve} is a complete reduced Noetherian local $\CC$-algebra $R$ of dimension one. Therefore, its normalization $\ol R$ is finite, and by \emph{splitting of normalization} it is the product of finitely many discrete valuation rings:
\[
\begin{tikzcd}[row sep=large,column sep=huge]
R \ar[rr,hook] \ar[rd,hook] \ar[rrdd,hook,dashed,bend right,"\phi"'] &&\ol R \ar[d,"\cong"] \\
&\displaystyle{\prod_{\fracid p \in \Min \br{R}}} R / \fracid p \ar[r,hook] &\displaystyle{\prod_{\fracid p \in \Min \br{R}}} \ol{R / \fracid p} \ar[d,"\cong"] \\
&&\displaystyle{\prod_{\fracid p \in \Min \br{R}}} \CC \ssbr{t_{\fracid p}}.
\end{tikzcd}
\]

This yields a \emph{parametrization} $\phi$ of the curve. Since the valuation of each $\CC \ssbr{t_{\fracid p}}$ is the order $\ord_{t_{\fracid p}}$, it furthermore allows for the definition of the classical \emph{semigroup of values}
\[
\Gamma_R = \ord \circ \phi \br{R^\reg} \subset \NN^{\Min \br{R}}.
\]

When Delgado characterized Gorenstein algebroid curves by symmetry of their value semigroups \cite{man.math.61.285}, he identified the following important properties:
\begin{enumerate}[label=(S\arabic*),ref=(S\arabic*)]
\setcounter{enumi}{-1}
\item\label{S0}
There is a $\gamma \in \Gamma_R$ such that $\gamma + \NN^{\Min \br{R}} \subset \Gamma_R$.
\item\label{S1}
If $\alpha, \beta \in \Gamma_R$, then $\inf \br{\alpha, \beta} = \br{\min \cbr{\alpha_{\fracid p}, \beta_{\fracid p}}}_{\fracid p \in \Min \br{R}} \in \Gamma_R$.
\item\label{S2}
For any $\alpha, \beta \in \Gamma_R$ with $\alpha_{\fracid p} = \beta_{\fracid p}$ for some $\fracid p \in \Min \br{R}$, there is an $\varepsilon \in \Gamma_R$ such that $\varepsilon_{\fracid p} > \alpha_{\fracid p} = \beta_{\fracid p}$ and $\varepsilon_{\fracid q} \geq \min \cbr{\alpha_{\fracid q}, \beta_{\fracid q}}$ for all $\fracid q \in \Min \br{R}$, where equality is obtained whenever $\alpha_{\fracid q} \ne \beta_{\fracid q}$.
\end{enumerate}

The reason behind \ref{S0} is that $\ol R$ is finite as an $R$-module. For \ref{S1} consider the sum of two elements $x,y \in R$ with values $\alpha$ and $\beta$ such that in no component of $\phi \br{x} + \phi \br{y}$ a term of least order is canceled. Complementary, for \ref{S2} we choose $x$ and $y$ such that there is cancellation in component $\fracid p$.

Barucci, D'Anna, and Fr\"oberg considered conditions \ref{S0}, \ref{S1}, and \ref{S2} as defining properties of the so-called \emph{good semigroups} \cite{j.pure.app.alg.147.215} which will be treated in Section~\ref{109}. Before, the goal of this section is to enlarge the class of available rings which admit a semigroup of values satisfying the above properties. This will lead to the notion of \emph{admissible rings} (see Definition~\ref{030}).


\subsection{(Regular) Fractional Ideals}
\label{114}


We begin with recalling some basic properties of fractional ideals, i.e.\ \enquote{ideals with denominators}, for later reference.


\begin{definition}
\label{036}
Let $R$ be a ring with total ring of fractions $Q_R$.
\begin{deflist}
\item\label{036a}
An $R$-submodule $\fracid{I}$ of $Q_R$ is called a \emph{fractional ideal} of $R$ if there is an $x \in R^\reg$ such that $x \fracid{I} \subset R$.
\item\label{036b}
A fractional ideal $\fracid{I}$ of $R$ is called \emph{regular} if $\fracid{I}^\reg \ne \emptyset$.
\item\label{036c}
We denote the set of regular fractional ideals of $R$ by $\RFI_R$.
\end{deflist}
\end{definition}


In the following, we collect some well-known properties of fractional ideals.


\begin{remark}
\label{102}
Let $R$ be a ring, let $x \in Q_R^\reg$, and let $\fracid{I}$, $\fracid{I}'$, $\fracid{J}$, $\fracid{J}'$, and $\fracid{H}$ be $R$-submodules of $Q_R$.
\begin{statements}
\item\label{102a}
If $R$ is Noetherian, then $\fracid I$ is a fractional ideal of $R$ if and only if it is finitely generated.
\item\label{102b}
$\br{x \fracid{I}} : \fracid{J} = x \br{\fracid{I} : \fracid{J}} = \fracid{I} : \br{x^{-1} \fracid{J}}$.
\item\label{102c}
If $\fracid{I} \subset \fracid{I}'$ and $\fracid{J} \subset \fracid{J}'$, then $\fracid{I} : \fracid{J}' \subset \fracid{I} : \fracid{J} \subset \fracid{I}' : \fracid{J}$.
\item\label{102d}
$\fracid{I} : R = \fracid{I}$.
\item\label{102e}
If $\fracid{I}, \fracid{J} \in \RFI_R$, then $\fracid{I} \fracid{J}, \fracid{I} : \fracid{J} \in \RFI_R$.
\end{statements}
For proofs, see for example \cite[Lemma~2.3 and Proposition~2.7]{korell2018}.
\end{remark}


Let $R$ be a ring. As shown in \cite[Lemma~2.1]{herzog1971} and \cite[Lemma~3.1]{j.symb.comp.45.887}, for any two regular $R$-submodules $\fracid{I}$ and $\fracid{J}$ of $Q_R$ there is a natural $R$-module isomorphism
\[
\Hom_R \br{\fracid{I},\fracid{J}} \to \fracid{J} : \fracid{I}, \quad\phi \mapsto \frac{\phi \br{x}}{x},
\]
which is independent of the choice of a regular element $x \in \fracid{I}^\reg$. In particular, any $\phi \in \Hom_R \br{\fracid{I}, \fracid{J}}$ is multiplication by an element of $\fracid{J} : \fracid{I}$, and it can be extended uniquely to an endomorphism of $Q_R$.


Together with Remark~\ref{102}.\ref{102e} and the Cayley--Hamilton Theorem (see \cite[Theorem~4.3]{eisenbud1995}) this yields the following.


\begin{proposition}
\label{086}
\pushQED{\qed}
Let $R$ be a ring, and let $\fracid{I} \in \RFI_R$. Then $\fracid{I} : \fracid{I}$ is an integral extension of $R$ in $Q_R$.
\qedhere
\end{proposition}


\begin{definition}
\label{079}
Let $R$ be a ring, and let $\fracid{I} \in \RFI_R$. The \emph{conductor} of $\fracid{I}$ is $\fracid{C}_R = \fracid{I} : \ol R$.
\end{definition}


\subsection{Discrete Valuation Rings}
\label{115}


In this section, let $Q$ be a ring with $Q^\reg = Q^\ast$ having a \emph{large Jacobson radical}, i.e.\ any prime ideal of $Q$ containing the Jacobson radical is maximal (see \cite[Proposition~19]{can.j.math.26.412} for equivalent characterizations).


\begin{definition}
\label{069}\
\begin{deflist}
\item\label{069a}
A \emph{valuation ring} of $Q$ is a proper subring $V$ of $Q$ such that the set $Q \setminus V$ is multiplicatively closed.
\item\label{069b}
Let $V$ be a valuation ring of $Q$. Then for any subring $R$ of $V$ with $Q_R = Q$ we call $V$ a \emph{valuation ring over $R$}.
\item\label{069c}
If $R$ is a subring of $Q$ with $Q_R = Q$, the set of valuation rings of $Q$ over $R$ is denoted by $\VR_R$.
\end{deflist}
\end{definition}


\begin{definition}
\label{072}
A valuation ring $V$ of $Q$ with regular maximal ideal $\fracid{m}_V$ is called a \emph{discrete valuation ring} if $\fracid{m}_V$ is finitely generated.
\end{definition}


\begin{definition}
\label{073}
Let $A$ be a ring. A \emph{discrete valuation} of $A$ is a surjective map $\nu \colon A \to \ZZ_\infty = \ZZ \cup \cbr{\infty}$ satisfying $\nu \br{x y} = \nu \br{x} + \nu \br{y}$ and $\nu \br{x + y} \geq \min \cbr{\nu \br{x}, \nu \br{y}}$ for every $x,y \in \ZZ_\infty$. Here we set $x + \infty = \infty + x = \infty$ and $\infty \geq x$ for all $x \in \ZZ_\infty$.
\end{definition}


\begin{proposition}
\label{075}
Let $Q$ be a ring with $Q^\reg = Q^\ast$, and let $V$ be a discrete valuation ring of $Q$. Then there is a discrete valuation $\nu_V$ of $Q$ such that the following hold.
\begin{statements}
\item\label{075a}
$V = \cbr{x \in Q \mid \nu_V \br{x} \geq 0}$, $\fracid{m}_V = \cbr{x \in Q \mid \nu_V \br{x} > 0}$, and $V^\ast = \cbr{x \in Q^\reg \mid \nu_V \br{x} = 0}$.
\item\label{075b}
The regular maximal ideal $\fracid{m}_V$ is generated by any $t \in Q^\reg$ with $\nu_V \br{t} = 1$. Such an element is called a \emph{uniformizing parameter} $V$.
\item\label{075c}
Let $t$ be a uniformizing parameter for $V$. Then for any $x \in Q^\reg$ there is a unique $a \in V^\ast$ such that $x = a t^{\nu_V \br{x}}$.
\end{statements}
\end{proposition}
\begin{proof}
See \cite[Propositions~2.22, 2.23, and D.13 and Corollary~2.25]{korell2018}.
\end{proof}


\subsection{Admissible Rings}


The theory of discrete valuation rings provides a general framework for the definition of value semigroups. So let $R$ be a one-dimensional equidimensional semilocal Cohen--Macaulay ring with total ring of fractions $Q_R$. Then the set $\VR_R$ is non-empty and finite, and it contains only discrete valuation rings. Moreover, $\ol R = \bigcap_{V \in \VR_R} V$ (see \cite[Chapter~II, Theorem~2.11]{kiyek2004}). We write $\nu_V$ for the discrete valuation $Q_R \to \ZZ_\infty$ corresponding to $V \in \VR_R$ (see Proposition~\ref{075}). Moreover, we denote
\[
\nu_R = \br{\nu_V}_{V \in \VR_R} \colon Q_R \to \br{\ZZ_\infty}^{\VR_R}.
\]

\begin{definition}
\label{028}
Let $R$ be a one-dimensional equidimensional semilocal Cohen--Macaulay ring.
\begin{deflist}
\item\label{028a}
The \emph{semigroup of values} (or \emph{value semigroup}) of $R$ is
\[
\Gamma_R = \nu_R \br{R^\reg} \subset \NN^{\VR_R}.
\]
\item\label{028b}
For a regular fractional ideal $\fracid{I}$ of $R$ we call
\[
\Gamma_\fracid{I} = \nu_R \br{\fracid{I}^\reg} \subset \ZZ^{\VR_R}
\]
the \emph{value semigroup ideal} of $\fracid{I}$.
\end{deflist}
\end{definition}

However, if we require value semigroups to satisfy conditions \ref{S0} -- \ref{S2}, we need to further restrict the available rings. For example, \ref{S0} means that the normalization is finite. In order to obtain sums with or without cancellation of values we need sufficiently large and compatible residue fields. As explained in \cite[Section~3]{j.commut.algebra.11.81} this leads to the following definition.


\begin{definition}
\label{030}
A one-dimensional equidimensional semilocal Cohen--Macaulay ring $R$ is called \emph{admissible} if the following hold.
\begin{deflist}
\item
$R$ is \emph{analytically reduced}, i.e.\ $\wh R$ is reduced.
\item
$R$ is \emph{residually rational}, i.e.\ $R / \fracid{m} = \ol R / \fracid{n}$ for any $\fracid{m} \in \Max \br{R}$ and for every $\fracid{n} \in \Max \br{\ol R}$ with $\fracid{n} \cap R = \fracid{m}$.
\item
$R$ has \emph{large residue fields}, i.e.\ $\vbr{R / \fracid{m}} \geq \vbr{\VR_{R_\fracid{m}}}$ for all $\fracid{m} \in \Max \br{R}$.
\end{deflist}
\end{definition}


Integral extensions of local ring are not necessarily local. For example, if $R$ is reduced, then $\vbr{\Max \br{\ol R}} = \vbr{\Min \br{\wh R}}$. Therefore, we allow algebroid curves to be semilocal. Since they are complete, an algebroid curve is the product of finitely many local algebroid curves.


\begin{definition}
\label{074}
Let $\KK$ be a field. An \emph{algebroid curve} over $\KK$ is a complete equidimensional reduced Noetherian semilocal $\KK$-algebra $R$ of dimension one such that  all residue fields of $R$ are isomorphic to $\KK$ (under the canonical surjections $R \to R / \fracid{m}$ for $\fracid{m} \in \Max \br{R}$). We call an algebroid curve \emph{admissible} if it is an admissible ring.
\end{definition}


\begin{proposition}
\label{076}
Let $\KK$ be a field, and let $R$ be an algebroid curve over $\KK$. If $\vbr{\KK} \geq \vbr{\Min \br{R}}$, then $R$ is admissible.
\end{proposition}
\begin{proof}
See \cite[Proposition~3.41]{korell2018}.
\end{proof}


\begin{proposition}
\label{029}
Let $R$ be a ring, and let $A$ be an integral extension of $R$ in $Q_R$.
\begin{statements}
\item\label{029a}
If $R$ is admissible, then $A$ is admissible with $\VR_R = \VR_A$ and $\nu_R = \nu_A$. In particular, the value semigroup ideal of $A$ as a regular fractional ideal of $R$ coincides with its semigroup of values as an admissible ring.
\item\label{029b}
If $R$ is an (admissible) algebroid curve over a field $\KK$, then $A$ is an (admissible) algebroid curve over $\KK$.
\end{statements}
\end{proposition}
\begin{proof}
See \cite[Theorem~3.45]{korell2018}.
\end{proof}


\begin{proposition}
\label{078}
Let $\KK$ be a field, and let $R$ be an algebroid curve over $\KK$. For every $\fracid{p} \in \Min \br{R}$ and any uniformizing parameter $t_\fracid{p}$ for the discrete valuation ring $\ol{R / \fracid{p}}$ there is an isomorphism
\begin{align*}
\phi_\fracid{p} \colon Q_R / \fracid{p} Q_R = Q_{R / \fracid{p}} &\to \KK \ssbr{T_\fracid{p}} \sbr{T_\fracid{p}^{-1}}, \\
t_\fracid{p} \mapsto T_\fracid{p},
\end{align*}
with $\phi_\fracid{p} \br{\ol R / \fracid{p}} = \KK \ssbr{t_\fracid{p}}$. This induces an isomorphism
\[
\phi \colon Q_R = \prod_{\fracid{p} \in \Min \br{R}} Q_R / \fracid{p} Q_R \to \prod_{\fracid{p} \in \Min \br{R}} \KK \ssbr{T_\fracid{p}} \sbr{T_\fracid{p}^{-1}}
\]
satisfying
\[
\nu_R = \ord_T \circ \phi.
\]
\end{proposition}
\begin{proof}
See \cite[Theorems~3.44 and A.74]{korell2018}.
\end{proof}


\subsection{Blowups and Stable Ideals}


Let again $R$ be a one-dimensional semilocal Cohen--Macaulay ring. \emph{Blowing up} a regular fractional ideal $\fracid I$ of $R$, we obtain a sequence of integral extensions of $R$ (see Proposition~\ref{086})
\begin{equation}
\label{103}
R \subset \fracid{I} : \fracid{I} \subset \fracid{I}^2 : \fracid{I}^2 \subset \ldots.
\end{equation}


Note that by definition there is a regular element $x \in Q_R^\reg$ such that $x \fracid{I}$ is a regular ideal of $R$. So Remark~\ref{102}.\ref{102b} yields for any $n \in \NN$
\[
\br{x \fracid{I}}^n : \br{x \fracid{I}}^n = \br{x^n \fracid{I}^n} : \br{x^n \fracid{I}^n} = x^n \br{x^n}^{-1} \br{\fracid{I}^n : \fracid{I}^n} = \fracid{I}^n : \fracid{I}^n.
\]
Thus, we may assume that $\fracid{I}$ is a regular ideal of $R$.


We assign a name to the union of the integral extensions in \eqref{103}.


\begin{definition}
\label{099}
Let $R$ be a one-dimensional semilocal Cohen--Macaulay ring, and let $\fracid{I}$ be a regular fractional ideal of $R$. The \emph{blowup} of $\fracid I$ is the ring $R^\fracid{I} = \bigcup_{n \in \NN} \br{\fracid{I}^n : \fracid{I}^n}$.
\end{definition}


In fact, the sequence \eqref{103} stabilizes after finitely many steps.


\begin{proposition}
\label{100}
Let $R$ be a one-dimensional semilocal Cohen--Macaulay ring, and let $\fracid{I}$ be a regular fractional ideal of $R$.
\begin{statements}
\item\label{100a}
$R^\fracid{I}$ is a finitely generated $R$-module, and $R^\fracid{I} = \fracid{I}^n : \fracid{I}^n$ for all sufficiently large $n$.
\item\label{100b}
There is a regular element $x \in R^\fracid{I}$ such that $\fracid{I} R^\fracid{I} = x R^\fracid{I}$.
\end{statements}
\end{proposition}
\begin{proof}
See \cite[Proposition~1.1]{amer.j.math.93.649}.
\end{proof}


It is of particular interest when the sequence \eqref{103} stabilizes at the first power of $\fracid I$.


\begin{definition}
\label{101}
Let $R$ be a one-dimensional semilocal Cohen--Macaulay ring. A regular fractional ideal $\fracid{I}$ of $R$ is called \emph{stable} if $R^\fracid{I} = \fracid{I} : \fracid{I}$.
\end{definition}


\begin{proposition}
\label{104}
Let $R$ be a one-dimensional semilocal Cohen--Macaulay ring. A regular fractional ideal $\fracid{I}$ of $R$ is stable if and only if there is an $x \in \fracid{I}$ such that $x \br{\fracid{I} : \fracid{I}} = \fracid{I}$.
\end{proposition}
\begin{proof}
If $\fracid{I}$ is stable, then by Proposition~\ref{100}.\ref{100b} there is an $x \in R^\fracid{I}$ such that $x \br{\fracid{I} : \fracid{I}} = \fracid{I} \br{\fracid{I} : \fracid{I}} = \fracid{I}$. Since $1 \in \fracid I : \fracid I$, this implies $x \in \fracid I$.

Conversely, suppose that there is an $x \in \fracid{I}$ such that $x \br{\fracid{I} : \fracid{I}} = \fracid{I}$. Then $\fracid{I}^2 = x \br{\fracid{I} : \fracid{I}} \fracid{I} = x \fracid{I}$. Hence, $\fracid{I}$ is stable by \cite[Lemma~1.11(i)]{amer.j.math.93.649}.
\end{proof}


\section{Good Semigroups}
\label{109}


Based on conditions \ref{S0}, \ref{S1}, and \ref{S2}, Barucci, D'Anna, and Fr\"oberg introduced the notion of \emph{good semigroups} \cite{j.pure.app.alg.147.215}. Before defining good semigroups, we describe these conditions more generally.


\begin{definition}
\label{005}
Let $I$ be a finite set, and consider on $\ZZ^I$ the natural partial order. We denote $N = \cbr{\alpha \in \ZZ^I \mid \alpha \geq \zero}$. For a subset $E$ of $\ZZ^I$ we consider the following properties (see \cite[Section~1]{man.math.61.285} and \cite[Section~2]{comm.algebra.25.2939}).
\begin{enumerate}[label=(E\arabic*),ref=(E\arabic*)]
\setcounter{enumi}{-1}
\item\label{E0}
There exists an $\alpha \in \ZZ^I$ such that $\alpha + N \subset E$.
\item\label{E1}
If $\alpha, \beta \in E$, then $\inf \br{\alpha, \beta} = \br{\min \cbr{\alpha_i, \beta_i}}_{i \in I} \in E$.
\item\label{E2}
For any $\alpha, \beta \in E$ with $\alpha_j = \beta_j$ for some $j \in I$, there exists an $\varepsilon \in E$ such that $\varepsilon_j > \alpha_j = \beta_j$ and $\varepsilon_i \ge \min \cbr{\alpha_i, \beta_i}$ for all $i \in I$, where equality is obtained whenever $\alpha_i \ne \beta_i$.
\end{enumerate}
\end{definition}


We are now equipped to introduce good semigroups.


\begin{definition}
\label{096}
Let $S$ be a submonoid of $\ZZ^I$, where $I$ is a finite set. Then $S$ is called a \emph{good semigroup} if it satisfies the following.
\begin{deflist}
\item\label{096a}
The unique minimal element of $S$ is $\zero$, i.e.\ any non-zero element of $S$ is comparable to and larger than $\zero$ with respect to the natural partial order on $\ZZ^I$.
\item\label{096b}
Properties~\ref{E1}, \ref{E1}, and \ref{E2} hold for $S$.
\end{deflist}
Throughout this article, we refer to $I$ without further explanation.
\end{definition}


These semigroups have a \enquote{tropical} semiring structure. In order to view this structure parallel to the structure of admissible rings, we restrict our considerations to certain semigroup ideals corresponding to the value sets of fractional ideals.


\begin{definition}
\label{031}
Let $S$ be a good semigroup.
\begin{deflist}
\item\label{031a}
We set $\ol S = \cbr{\alpha \in \ZZ^I \mid \alpha \geq \zero} \cong \NN^I$.
\item\label{031b}
A good semigroup $S$ is said to be a \emph{numerical semigroup} if $\vbr{I} = 1$.
\item\label{031c}
A \emph{semigroup ideal} of a good semigroup $S$ is a non-empty subset $E$ of $\ZZ^I$ such that $E + S \subset E$ and $\alpha + E \in S$ for some $\alpha \in S$.
\item\label{031d}
A \emph{good semigroup ideal} of a good semigroup $S$ is a semigroup ideal $E$ of $S$ satisfying properties~\ref{E1} and \ref{E2}.
\item\label{031e}
For a good semigroup $S$ we denote by $\GSI_S$ the set of all good semigroup ideals of $S$.
\end{deflist}
\end{definition}


In fact, the definitions of good semigroups and good semigroup ideals describe value sets of admissible rings and their fractional ideals.


\begin{proposition}
\label{091}
Let $R$ be an admissible ring. Then $\Gamma_R$ is a good semigroup with $I = \VR_R$, and $\Gamma_\fracid{I} \in \GSI_{\Gamma_R}$ for every $\fracid{I} \in \RFI_R$.
\end{proposition}
\begin{proof}
See \cite[Remark~4.1.2.(d)]{j.commut.algebra.11.81}.
\end{proof}


However, value semigroups of admissible rings are not characterized as good semigroups. An example of a good semigroup which is not the value semigroups of an admissible ring can be found in \cite[Example~2.16]{j.pure.app.alg.147.215}.


\begin{remark}
\label{032}\
\begin{statements}
\item\label{032d}
Let $S$ be a good semigroup in $\ZZ^I$. Then its group of differences is $\ZZ^I$.
\item\label{032a}
If $S$ is a good semigroup, any semigroup ideal $E$ of $S$ satisfies property \ref{E0} since $S$ does and $E + S \subset E$.
\item\label{032b}
If $S$ is a numerical semigroup, then $E \in \GSI_S$ for every semigroup ideal $E$ of $S$.
\item\label{032c}
If $S$ and $S'$ are good semigroups with $S \subset S' \subset \ol S$, then $\ol{S'} = \ol S$. It follows that $\GSI_{S'} \subset \GSI_S$, and, in particular, $S' \in \GSI_S$.
\end{statements}
See~\cite[Remark~4.1.2]{j.commut.algebra.11.81}.
\end{remark}


\begin{definition}
\label{050}
Let $S$ be a good semigroup, and let $J \subset I$. We denote by $\pr_J \colon \ZZ^I \to \ZZ^J$ the projection map. For a subset $E$ of $\ZZ^I$ we call $\pr_J \br{E}$ the \emph{projection} of $E$, and we denote it by $E_J$. We write $E_i = E_{\cbr{i}}$ for $i \in I$.
\end{definition}


In fact, the projection of a good semigroup $S$ onto a subset $J$ of $I$ is again a good semigroup, and $E_J \in \GSI_{S_J}$ for every $E \in \GSI_S$, see \cite[Proposition~4.59]{korell2018} and \cite[Proposition~2.2]{j.pure.app.alg.147.215}.


\begin{definition}
\label{051}\
\begin{deflist}
\item\label{051a}
Let $R$ be an admissible ring, and let $\fracid{I}$ be an $R$-submodule of $Q_R$. For $\alpha \in \ZZ^I$ we set
\[
\fracid{I}^\alpha = \cbr{x \in \fracid{I} \mid \nu \br{x} \geq \alpha} = \fracid I \cap \br{Q_R}^\alpha.
\]
\item\label{051b}
Let $S$ be a good semigroup, and let $E \subset \ZZ^I$. For $\alpha \in \ZZ^I$ we set
\[
E^\alpha = \cbr{\beta \in E \mid \beta \geq \alpha} = E \cap \br{\ZZ^I}^\alpha.
\]
\end{deflist}
\end{definition}


\begin{remark}
\label{025}\
\begin{statements}
\item\label{025a}
Let $R$ be an admissible ring, and let $\fracid{I} \in \RFI_R$. For any $\alpha \in \ZZ^I$ the set $\fracid{I}^\alpha$ is a regular fractional ideal of $R$ with $\Gamma_{\fracid{I}^\alpha} = \br{\Gamma_\fracid{I}}^\alpha$.
\item\label{025b}
Let $S$ be a good semigroup, and let $E \in \GSI_S$. For any $\alpha \in \ZZ^I$ the set $E^\alpha$ is a good semigroup ideal of $S$.
\end{statements}
\end{remark}


\subsection{Localization}


Considering only local algebroid curves, Delgado found that $\zero$ is the only element of the value semigroup with a zero component. Accordingly, a good semigroup $S$ has been called \emph{local} if $\zero$ is the only element of $S$ with a zero component \cite{j.pure.app.alg.147.215,j.commut.algebra.11.81,korell2018}.

In a local good semigroup $S$, the set $M_S = S \setminus \cbr{\zero}$ is a good semigroup ideal, and it is called the \emph{maximal ideal} of $S$. In fact, an admissible ring $R$ is local if and only if its semigroup of values $\Gamma_R$ is local, and if $\fracid m$ is the maximal ideal of $R$, then $\Gamma_{\fracid m} = M_{\Gamma_R}$ (see \cite[page~6]{j.pure.app.alg.147.215}). We extend this definition to an intrinsic description of maximal ideals of good semigroups (see Definition~\ref{007}).

It is known that a good semigroup is semilocal in the sense that it decomposes uniquely into a product of finitely many local good semigroups. Moreover, this decomposition is compatible with localizations of admissible rings: If $R$ is an admissible ring, then $\Gamma_R = \prod_{\fracid m \in \Max \br{R}} \Gamma_{R_{\fracid m}}$, see \cite[§~1.1]{j.pure.app.alg.147.215} and \cite[Theorems~3.2.2 and 4.1.6]{j.commut.algebra.11.81}.


\begin{definition}
\label{007}
Let $S$ be a good semigroup.
\begin{enumerate}[label=(\arabic*),ref=(\arabic*)]
\item\label{007a}
A good semigroup ideal $M$ of $S$ with $M \subsetneq S$ is called \emph{maximal} if for every good semigroup ideal $E$ of $S$ with $M \subset E \subsetneq S$ we have $E = M$.
\item\label{007b}
We denote the set of maximal ideals of $S$ by $\Max \br{S}$.
\item\label{007c}
$S$ is called \emph{local} if it has a unique maximal ideal $M_S$.
\end{enumerate}
\end{definition}


\begin{remark}
\label{095}
Any numerical semigroup is a local good semigroup.
\end{remark}


It is easy to see that any good semigroup is \enquote{semilocal}:


\begin{lemma}
\label{068}
Let $S$ be a good semigroup. Then $\Max \br{S} = \cbr{S^{\ee_i} \mid i \in I}$.
\end{lemma}
\begin{proof}
Let $E$ be a good semigroup ideal of $S$ with $E \subsetneq S$. Since $E$ satisfies property~\ref{E1}, there is an $i \in I$ such that $\alpha_i > 0$ for all $\alpha \in E$ (otherwise $\zero \in E$). This implies $E \subset S^{\ee_i}$. Therefore, we have $\Max \br{S} = \cbr{S^{\ee_i} \mid i \in I}$ (see Remark~\ref{025}.\ref{025b}).
\end{proof}


\begin{remark}
\label{107}
In particular, Lemma~\ref{068} relates our definition of maximal ideals and local good semigroups to the one used for example in~\cite{j.pure.app.alg.147.215,j.commut.algebra.11.81,korell2018}. In fact, a good semigroup $S$ is local if and only if its unique maximal ideal is $M_S = \bigcap_{i \in I} S^{\ee_i} = S^\one = S \setminus \cbr{\zero}$.
\end{remark}


With a notion of maximal ideals of good semigroups, we could use the semiring structure to define localizations in the \enquote{usual} way. It turns out, however, that it suffices to consider projections.


\begin{definition}
\label{045}
Let $S$ be a good semigroup, and let $E \in \GSI_S$. For a maximal ideal $M$ of $S$ we define the \emph{localization} of $E$ with respect to $M$ as $E_M = E_{I \setminus I_M}$, where $E_{I \setminus I_M}$ is the projection of Definition~\ref{050} with $I_M = \cbr{i \in I \mid 0 \in M_i}$.
\end{definition}


With the above definitions, the decompositions of good semigroups, respectively value semigroups, as described in \cite[Corollary~3.2.3 and Theorem~4.1.6]{j.commut.algebra.11.81} can be reformulated in terms of localizations. The details of this reformulation are left to the reader.


\begin{theorem}
\label{105}
\pushQED{\qed}
Let $S$ be a good semigroup, and let $R$ be an admissible ring.
\begin{statements}
\item\label{105a}
For every maximal ideal $M$ of $S$ the good semigroup $S_M$ is local with maximal ideal $M_M$.
\item\label{105b}
Every good semigroup ideal $E$ of $S$ decomposes as $E = \prod_{M \in \Max \br{S}} E_M$.
\item\label{105c}
There is a bijection
\[
\Max \br{R} \to \Max \br{\Gamma_R}, \quad \fracid{m} \mapsto \Gamma_\fracid{m}.
\]
(note that maximal ideals of $R$ must be regular as $R$ is equidimensional, see \cite[Proposition~B.27]{korell2018}). In particular, $R$ is local if and only if $\Gamma_R$ is local.
\item\label{105d}
For $\fracid I \in \RFI_R$ and $\fracid m \in \Max \br{R}$, we have $\Gamma_{\fracid I_{\fracid m}} = \br{\Gamma_\fracid I}_{\Gamma_{\fracid m}}$.
\end{statements}
In particular, for every $\fracid I$ in $\RFI_R$ we have $\Gamma_{\fracid I} = \prod_{M \in \Max \br{\Gamma_R}} \br{\Gamma_{\fracid I}}_M = \prod_{\fracid m \in \Max \br{R}} \br{\Gamma_{\fracid I}}_{\Gamma_{\fracid m}}$.
\qedhere
\end{theorem}


\subsection{Differences}


Products and quotients are important operations for fractional ideals. The obvious combinatorial counterparts appear to be sums and differences.


\begin{definition}
\label{054}
Let $S$ be a good semigroup, and let $E, F \subset _S$. The \emph{difference} of $E$ and $F$ is
\[
E - F = \cbr{\alpha \in \ZZ^I \mid \alpha + F \subset E}.
\]
\end{definition}


However, good semigroup ideals are not closed under these operations.


\begin{remark}
\label{055}
Let $S$ be a good semigroup, and let $E$ and $F$ be two semigroup ideals of $S$. Then $E + F$ and $E - F$ are semigroup ideals of $S$ (see \cite[Lemma~4.18]{korell2018}). However, $E, F \in \GSI_S$ does in general neither imply $E + F \in \GSI_S$ (see Example~\ref{087} below) nor $E - F \in \GSI_S$ (see Example~\ref{088} below).
\end{remark}


In particular, sums and differences are not compatible with their algebraic counterparts.


\begin{remark}
\label{022}
Let $R$ be an admissible ring, and let $\fracid{I}, \fracid{J} \in \RFI_R$.
\begin{enumerate}[label=(\arabic*),ref=(\arabic*)]
\item\label{022a}
If $\fracid{I} \subset \fracid{J}$, then $\Gamma_\fracid{I} \subset \Gamma_\fracid{J}$.
\item\label{022b}
$\Gamma_\fracid{I} + \Gamma_\fracid{J} \subset \Gamma_{\fracid{I} \fracid{J}}$. In general, however, this inclusion is not an equality, see Example~\ref{087} below.
\item\label{022c}
$\Gamma_{\fracid{I}:\fracid{J}} \subset \Gamma_{\fracid{I}} - \Gamma_\fracid{J}$. In general, however, this inclusion is not an equality, see Example~\ref{088} below.
\end{enumerate}
See \cite[Remark~3.1.10]{j.commut.algebra.11.81}.
\end{remark}


\begin{figure}
\begin{center}
\begin{tikzpicture}[inner sep=1.5,scale=0.5]
\draw[->] (0,0) -- (0,6);
\draw[->] (0,0) -- (10,0);

\foreach \i in {0,...,9} \foreach \j in {0,...,5} \draw (\i,\j) node[shape=circle,draw,fill=white] {};
\draw (0,0) node[shape=circle,draw,fill=black] {};
\foreach \i in {0,...,6} \foreach \j in {0,...,4} \draw (3+\i,1+\j) node[shape=circle,draw,fill=black] {};

\draw (5,6) node {$\Gamma_R$};
\end{tikzpicture}\quad
\begin{tikzpicture}[inner sep=1.5,scale=0.5]
\draw[->] (0,0) -- (0,6);
\draw[->] (0,0) -- (10,0);

\foreach \i in {0,...,9} \foreach \j in {0,...,5} \draw (\i,\j) node[shape=circle,draw,fill=white] {};
\foreach \i in {0,...,4} \draw (2,1+\i) node[shape=circle,draw,fill=black] {};
\draw (3,1) node[shape=circle,draw,fill=black] {};
\foreach \i in {0,...,4} \foreach \j in {0,...,3} \draw (5+\i,2+\j) node[shape=circle,draw,fill=black] {};

\draw (5,6) node {$\Gamma_\fracid{I}$};
\end{tikzpicture}\medskip\\\noindent
\begin{tikzpicture}[inner sep=1.5,scale=0.5]
\draw[->] (0,0) -- (0,6);
\draw[->] (0,0) -- (10,0);

\foreach \i in {0,...,9} \foreach \j in {0,...,5} \draw (\i,\j) node[shape=circle,draw,fill=white] {};
\draw (3,1) node[shape=circle,draw,fill=black] {};
\foreach \i in {0,...,5} \foreach \j in {0,...,3} \draw (4+\i,2+\j) node[shape=circle,draw,fill=black] {};

\draw (5,6) node {$\Gamma_\fracid{J}$};
\end{tikzpicture}\quad
\begin{tikzpicture}[inner sep=1.5,scale=0.5]
\draw[->] (0,0) -- (0,6);
\draw[->] (0,0) -- (10,0);

\foreach \i in {0,...,9} \foreach \j in {0,...,5} \draw (\i,\j) node[shape=circle,draw,fill=white] {};
\draw (5,2) node[shape=circle,draw,fill=black] {};
\draw (6,2) node[shape=circle,draw,fill=black] {};
\foreach \i in {0,...,4} \foreach \j in {0,1,2} \draw (5+\i,3+\j) node[shape=circle,draw,fill=black] {};

\draw (5,6) node {$\Gamma_\fracid{I}+\Gamma_\fracid{J}$};
\end{tikzpicture}
\end{center}
\caption{The value semigroup (ideals) in Example~\ref{087}.}
\label{089}
\end{figure}
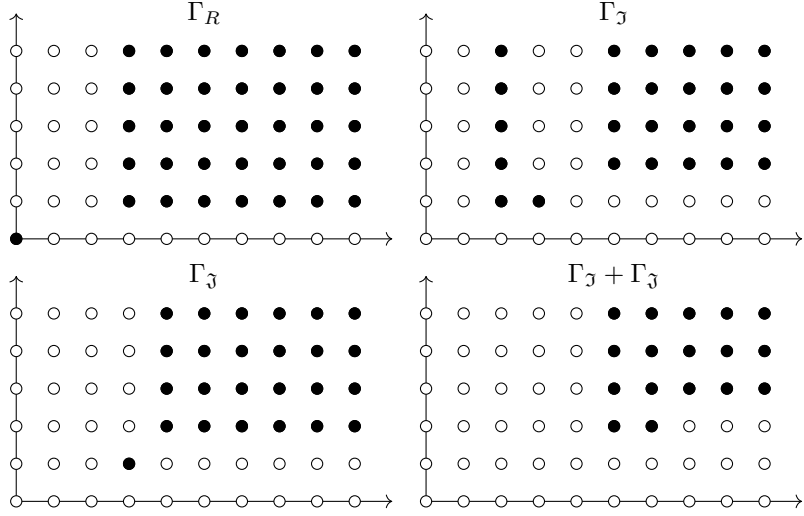


\begin{example}
\label{087}
Consider the admissible ring (see Proposition~\ref{076})
\[
R = \CC \left[\left[ \left( -t_1^4,t_2 \right), \left( -t_1^3,0 \right), \left( 0,t_2 \right), \left( t_1^5,0 \right) \right]\right] \subset \CC \left[\left[ t_1 \right]\right] \times \CC \left[\left[ t_2 \right]\right] = \ol R,
\] 
and the $R$-submodules of $Q_R$
\begin{align*}
\fracid{I} &= \left\langle \left( t_1^3, t_2 \right), \left( t_1^2, 0 \right) \right\rangle_R, \\
\fracid{J} &= \left\langle \left( t_1^3, t_2 \right), \left( t_1^4, 0 \right), \left( t_1^5, 0 \right) \right\rangle_R.
\end{align*}
Then $\fracid{I}, \fracid{J} \in \RFI_R$ (see Remark~\ref{102}.\ref{102a}). Moreover, Figure~\ref{089} shows that $\Gamma_\fracid{I} + \Gamma_\fracid{J}$ does not satisfy property~\ref{E2}. Thus, $\Gamma_\fracid{I} + \Gamma_\fracid{J} \subsetneq \Gamma_{\fracid{I} \fracid{J}}$ by Remark~\ref{022}.\ref{022b} and Proposition~\ref{091}.
\end{example}


\begin{figure}
\begin{center}
\begin{tikzpicture}[scale=0.4,inner sep=1.5]
\draw[->] (0,0) -- (0,17);
\draw[->] (0,0) -- (17,0);

\foreach \i in {0,...,16} \foreach \j in {0,...,16} \draw (\i,\j) node[shape=circle,draw,fill=white] {};

\draw (0,0) node[shape=circle,draw,fill=black] {};
\draw (6,6) node[shape=circle,draw,fill=black] {};
\draw (6,7) node[shape=circle,draw,fill=black] {};
\draw (7,6) node[shape=circle,draw,fill=black] {};
\foreach \i in {9,...,16} \draw (\i,9) node[shape=circle,draw,fill=black] {};
\foreach \i in {9,10,11} \draw (\i,10) node[shape=circle,draw,fill=black] {};
\draw (9,11) node[shape=circle,draw,fill=black] {};
\foreach \i in {12,...,16} \draw (10,\i) node[shape=circle,draw,fill=black] {};
\foreach \i in {12,...,16} \foreach \j in {12,...,16} \draw (\i,\j) node[shape=circle,draw,fill=black] {};

\draw (8.5,17.5) node {$\Gamma_R$};
\end{tikzpicture}\quad
\begin{tikzpicture}[scale=0.4,inner sep=1.5]
\draw (1,0) node[shape=circle,draw=white,fill=white] {};

\draw[->] (0,0) -- (0,17);
\draw[->] (0,0) -- (17,0);

\foreach \i in {0,...,16} \foreach \j in {0,...,16} \draw (\i,\j) node[shape=circle,draw,fill=white] {};

\draw (3,4) node[shape=circle,draw,fill=black] {};
\foreach \i in {3,4,5} \draw (\i,3) node[shape=circle,draw,fill=black] {};
\foreach \i in {6,...,16} \foreach \j in {6,...,16} \draw (\i,\j) node[shape=circle,draw,fill=black] {};

\draw (8.5,17.5) node {$\Gamma_{\fracid{m}_R} - \Gamma_{\fracid{m}_R}$};
\end{tikzpicture}
\end{center}
\caption{The value semigroup (ideals) in Example~\ref{088}, see~\cite[Example~3.3]{j.pure.app.alg.147.215}.}
\label{092}
\end{figure}


\begin{example}
\label{088}
Barucci, D'Anna and Fr\"oberg showed in \cite[Example~3.3]{j.pure.app.alg.147.215} that for the local admissible ring (see Figure~\ref{092}, Proposition~\ref{076} and Remark~\ref{107})
\[
R = \mathbb{C} \ssbr{x_1,\ldots, x_{11}}
\]
with $x_1 = \br{t_1^7,t_2^6}$, $x_2 = \br{t_1^6,t_2^7}$, $x_3 = \br{t_1^9,t_2^{11}}$, $x_4 = \br{t_1^{10},t_2^{10}}$, $x_5 = \br{t_1^{11},t_2^9}$, $x_6 = \br{t_1^{11},t_2^{10}}$, $x_7 = \br{t_1^{12},t_2^{12}}$, $x_8 = \br{t_1^{13},-t_2^{13}}$, $x_9 = \br{t_1^{20},t_2^{12}}$, $x_{10} = \br{t_1^{16},t_2^{20}}$, $x_{11} = \br{t_1^{12},t_2^{20}}$ with maximal ideal $\fracid{m}_R$ the difference $\Gamma_{\fracid{m}_R} - \Gamma_{\fracid{m}_R}$ does not satisfy property~\ref{E2}, see Figure~\ref{092}. Thus, $\Gamma_{\fracid{m}_R} - \Gamma_{\fracid{m}_R} \subsetneq \Gamma_{\fracid{m}_R : \fracid{m}_R}$ by Remark~\ref{022}.\ref{022c} and Proposition~\ref{091}.
\end{example}


\begin{remark}
\label{003}
Let $S$ be a good semigroup, and let $\alpha \in \ZZ^I$.
\begin{statements}
\item\label{003a}
Let $E$, $F$, and $G$ be semigroup ideals of $S$. Then
\[
\br{E - F} - G = E - \br{F + G} = \br{E - G} - F.
\]
\item\label{003b}
For any two semigroup ideals $E$ and $F$ of $S$ we have
\[
\br{\alpha + E} - F = \alpha + \br{E - F} = E - \br{- \alpha + F}.
\]
\item\label{003c}
Let $E$, $E'$, $F$, and $F'$ be semigroup ideals of $S$. If $E \subset E'$ and $F \subset F'$, then
\[
E - F' \subset E - F \subset E' - F.
\]
\item\label{003d}
For any $E \in \GSI_S$ we have $E - S = E$.
\item\label{003e}
The map $\GSI_S \to \GSI_S$, $E \mapsto \alpha + E$ is a bijection.
\end{statements}
\end{remark}


Let $S$ be a good semigroup, and let $E \in \GSI_S$. Since $E$ satisfies property~\ref{E1}, and since it is bounded from below, there is a unique minimal element of $E$ (see \cite[Lemma~4.12]{korell2018}). Moreover, the set $E - \ol S$ is a good semigroup ideal of $S$ (see \cite[Lemma~4.1.4]{j.commut.algebra.11.81}). This allows for the following definition.


\begin{definition}
\label{052}
Let $S$ be a good semigroup, and let $E \in \GSI_S$.
\begin{deflist}
\item\label{052a}
We denote the \emph{minimal element} of $E$ by $\mu_E$.
\item\label{052b}
The \emph{conductor ideal} of $E$ is $C_E = E - \ol S$.
\item\label{052c}
The \emph{conductor} of $E$ is $\gamma_E = \mu_{C_E} = \inf \cbr{\alpha \in \ZZ^I \mid \alpha + \ol S \subset E}$.
\item\label{052d}
We write $\tau_E = \gamma_E - \one$.
\end{deflist}
\end{definition}


\begin{remark}
\label{053}
Let $S$ be a good semigroup, and let $E \in \GSI_S$. Then
\begin{statements}
\item\label{053a}
$E \subset \mu_E + \ol S$, and
\item\label{053b}
$C_E = \gamma_E + \ol S \subset E$ (see \cite[Remark~4.27]{korell2018}).
\end{statements}
\end{remark}


\begin{lemma}
\label{004}
Let $S$ be a good semigroup, and let $E$ and $F$ be good semigroup ideals of $S$. Then $\gamma_{E - F} = \gamma_E - \mu_F$.
\end{lemma}
\begin{proof}
See \cite[Lemma~4.1.11]{j.commut.algebra.11.81}
\end{proof}


\begin{lemma}
\label{064}
Let $S$ be a good semigroup, and let $E \in \GSI_S$. For any $\alpha \in \ZZ^I$, we have $\alpha \in E$ if and only if $\inf \br{\alpha, \gamma_E} \in E$.
\end{lemma}
\begin{proof}
See \cite[Lemma~4.1.9]{j.commut.algebra.11.81}.
\end{proof}


\begin{lemma}
\label{081}
Let $R$ and $R'$ be admissible rings such that $\fracid{C}_{R'} \subset R \subset R' \subset Q_R$. Then
\begin{statements}
\item\label{081a}
$R' \in \RFI_R$, and
\item\label{081b}
if $\Gamma_R = \Gamma_{R'}$, then $R = R'$.
\end{statements}
\end{lemma}
\begin{proof}\
\begin{statements}
\item
Let $x \in \br{\fracid{C}_R'}^\reg \subset R^\reg$. Then $x R' \subset \fracid{C}_R \subset R$. Since $\emptyset \ne R^\reg \subset \br{R'}^\reg$, this yields $R' \subset \RFI_R$.
\item
Since $R' \in \RFI_R$ by \ref{081a}, and since $R \in \RFI_R$, this follows from \cite[Corollary~4.2.8]{j.commut.algebra.11.81}.
\qedhere
\end{statements}
\end{proof}


\begin{proposition}
\label{082}
Let $R$ be an admissible ring, and let $\fracid{I} \in \RFI_R$. Then $\fracid{C}_\fracid{I} = \br{Q_R}^{\gamma_{\Gamma_\fracid{I}}}$.
\end{proposition}
\begin{proof}
See \cite[Proposition~4.56]{korell2018}.
\end{proof}


\subsection{Stable Ideals}


Stable semigroup ideals were introduced by Barucci, D'Anna, and Fr\"oberg \cite{j.pure.app.alg.147.215}. The definition reflects the characterization of stable fractional ideals in Proposition~\ref{104}.


\begin{definition}
\label{010}
Let $S$ be a good semigroup. A good semigroup ideal $E$ of $S$ is called \emph{stable} if there is an $\alpha \in E$ such that $\alpha + \br{E - E} = E$. Note that then $\alpha = \mu_E$.
\end{definition}


Stability of a fractional ideal $\fracid I$ implies stability of the value semigroup ideal $\Gamma_{\fracid I}$. While the converse is in general not true, we have the following characterization.


\begin{proposition}
\label{011}
Let $R$ be an admissible ring, and let $\fracid{I} \in \RFI_R$. Then $\fracid{I}$ is stable if and only if $\Gamma_{\fracid{I} : \fracid{I}} = \Gamma_\fracid{I} - \Gamma_\fracid{I}$ and $\Gamma_\fracid{I}$ is stable.
\end{proposition}
\begin{proof}
See \cite[Proposition~3.14]{j.pure.app.alg.147.215}.
\end{proof}


\section{Canonical Ideals and Gorenstein Property}
\label{110}


As seen in the previous section, good semigroup ideals do not behave well under differences in general. However, extending results by J\"ager \cite{arch.math.29.504} and D'Anna \cite{comm.algebra.25.2939}, \emph{canonical semigroup ideals} were introduced in \cite{j.commut.algebra.11.81}.


\begin{definition}
\label{015}\
\begin{deflist}
\item\label{015a}
Let $R$ be a one-dimensional equidimensional Cohen--Macaulay ring. A regular fractional ideal $\fracid{K}$ is called a \emph{canonical ideal} of $R$ if $\fracid{K} : \br{\fracid{K} : \fracid{I}} = \fracid{I}$ for every $\fracid{I} \in \RFI_R$.
\item\label{015b}
Let $S$ be a good semigroup. A good semigroup ideal $S$ is called a \emph{canonical (semigroup) ideal} of $S$ if $K \subset E$ implies $K = E$ for all $E \in \GSI_S$ with $\gamma_E = \gamma_K$.
\end{deflist}
\end{definition}


It turns out that canonical semigroup ideals induce a duality on good semigroup ideals (see Theorem~\ref{018}) in analogy to that on fractional ideals. So if $S$ is a good semigroup, then a canonical semigroup ideal $K$ is characterized by the property $K - \br{K - E} = E$ for every good semigroup ideal $E$ of $S$. Moreover, the \emph{dual} $K - E$ of $E$ is in fact a good semigroup ideal.

Canonical fractional ideals and canonical semigroup ideals are compatible in the following sense: Canonical fractional ideals are characterized by their value semigroup ideals, and dualizing commutes with taking values.


\begin{theorem}
\label{016}
Let $R$ be an admissible ring.
\begin{statements}
\item\label{016a}
A regular fractional ideal $\fracid{K}$ of $R$ is canonical if and only if $\Gamma_\fracid{K}$ is a canonical ideal of $\Gamma_R$.
\item\label{016b}
If $\fracid{K}$ is a canonical ideal of $R$, then $\Gamma_{\fracid{K}: \fracid{I}} = \Gamma_\fracid{K} - \Gamma_\fracid{I}$ for every regular fractional ideal $\fracid{I}$ of $\fracid{K}$.
\end{statements}
\end{theorem}
\begin{proof}
See \cite[Theorems~5.3.2 and 5.3.4]{j.commut.algebra.11.81}.
\end{proof}


Before recalling the main properties of canonical semigroup ideals in Theorem~\ref{018} we need a few technical definitions in order to describe canonical semigroup ideals explicitly.


\begin{definition}
\label{039}
Let $I$ be a finite set, and let $\alpha \in \ZZ^I$.
\begin{deflist}
\item\label{039a}
For any subset $J$ of $I$ we set
\[
\Delta_J \br{\alpha} = \cbr{\beta \in \ZZ^s \mid \alpha_j = \beta_j \text{ for all } j \in J \text{ and } \alpha_i < \beta_i \text{ for all } i \in I \setminus J}.
\]
\item\label{039b}
We set
\[
\Delta \br{\alpha} = \bigcup_{i \in I} \Delta_i \br{\alpha},
\]
where $\Delta_i \br{\alpha} = \Delta_{\cbr{i}} \br{\alpha}$ for any $i \in I$.
\item\label{039c}
For a subset $E$ of $\ZZ^s$ and any subset $J$ of $I$ we denote
\[
\Delta_J^E \br{\alpha} = \Delta_J \br{\alpha} \cap E.
\]
Similarly, we write $\Delta_i^E \br{\alpha}$ for given $i \in I$ as well as $\Delta^E \br{\alpha}$ when considering the union.
\item\label{039d}
For any subset $J$ of $I$ we set
\[
\ol{\Delta}_J \br{\alpha} = \cbr{\beta \in \ZZ^s \mid \alpha_j = \beta_j \text{ for all } j \in J \text{ and } \alpha_i \leq \beta_i \text{ for all } i \in I \setminus J},
\]
and as in \ref{039a}, \ref{039b}, and \ref{039c} we define accordingly $\ol \Delta_i \br{\alpha}$, $\ol \Delta \br{\alpha}$, $\ol \Delta_J^E \br{\alpha}$, $\ol \Delta_i^E \br{\alpha}$, and $\ol \Delta^E \br{\alpha}$ for any $i \in I$ and for any subset $E$ of $\ZZ^I$.
\end{deflist}
\end{definition}


Turning Delgado's symmetry condition for value semigroups of Gorenstein algebroid curves \cite{man.math.61.285} into the description of a semigroup ideal, D'Anna characterized \enquote{normalized} canonical ideals of admissible rings by having the following value semigroup ideal \cite{comm.algebra.25.2939}.


\begin{definition}
\label{017}
Let $S$ be a good semigroup. The set
\[
K_S^0 = \cbr{\alpha \in \ZZ^I \mid \Delta^S \br{\tau_S - \alpha} = \emptyset}
\]
is called the \emph{(normalized) canonical (semigroup) ideal} of $S$.
\end{definition}


As for fractional ideals, canonical semigroup ideals induce a duality on good semigroup ideals, and they are characterized by this property.


\begin{theorem}
\label{018}
Every good semigroup has a canonical ideal. Moreover, for any $K \in \GSI_S$ the following are equivalent.
\begin{enumerate}[label=(\alph*),ref=(\alph*)]
\item\label{018a}
$K$ is a canonical ideal of $S$.
\item\label{018b}
There is an $\alpha \in \ZZ^I$ such that $\alpha + K = K_S^0$.
\item\label{018c}
For all $E \in \GSI_S$ we have $K - \br{K - E} = E$.
\end{enumerate}
If $K$ is a canonical ideal of $S$, then the following hold.
\begin{enumerate}[label=(\arabic*),ref=(\arabic*)]
\item\label{018d}
If $S \subset K \subset \ol S$, then $K = K_S^0$.
\item\label{018e}
$K - E \in \GSI_S$ for every $E \in \GSI_S$.
\item\label{018f}
If $S'$ is a good semigroup with $S \subset S' \subset \ol S$, then $K' = K - S$ is a canonical ideal of $S'$.
\end{enumerate}
\end{theorem}
\begin{proof}
See \cite[Theorem~5.2.6]{j.commut.algebra.11.81}.
\end{proof}


We are particularly interested in the case when an admissible ring, respectively a good semigroup, is its own canonical ideal.


\begin{definition}
\label{019}\
\begin{enumerate}[label=(\arabic*),ref=(\arabic*)]
\item\label{019a}
An admissible ring $R$ is called \emph{Gorenstein} if $R$ is a canonical ideal of itself (see \cite[Remark~5.1.4]{j.commut.algebra.11.81}).
\item\label{019b}
A good semigroup $S$ is called \emph{symmetric} if $S$ is a canonical ideal of itself, i.e.
\[
S = \cbr{\alpha \in \ZZ^I \mid \Delta^S \br{\tau_S - \alpha} = \emptyset}.
\]
\end{enumerate}
\end{definition}


\begin{proposition}
\label{020}
An admissible ring $R$ is Gorenstein if and only if $\Gamma_R$ is symmetric.
\end{proposition}
\begin{proof}
See \cite[Proposition~5.3.6]{j.commut.algebra.11.81}.
\end{proof}


In the following we collect some results on canonical ideals for later reference.


\begin{proposition}
\label{034}
Let $R$ be a one-dimensional equidimensional semilocal Cohen--Macaulay ring with canonical ideal $\fracid{K}$. A regular fractional ideal $\fracid{K}'$ of $R$ is a canonical ideal of $R$ if and only if there is an $x \in Q_R^\reg$ such that $\fracid{K}' = x \fracid{K}$.
\end{proposition}
\begin{proof}
See \cite[Satz~2.8]{herzog1971} and \cite[Proposition~5.1.5]{j.commut.algebra.11.81}.
\end{proof}


Canonical modules propagate from a local Cohen--Macaulay ring $R$ to a finite local $R$-algebra, see \cite[Theorem~3.3.7.(b)]{bruns1998}. As explained in \cite[Remark~5.14]{j.commut.algebra.11.81}, canonical ideals of one-dimensional equidimensional Cohen--Macaulay rings are canonical modules. The following lemma provides a direct proof of the general propagation formula for canonical ideals in dimension one without the assumption of locality.


\begin{lemma}
\label{035}
Let $R$ and $A$ be one-dimensional equidimensional Cohen--Macaulay rings such that $A \in \RFI_R$. If $\fracid{K}$ is a canonical ideal of $R$, then $\fracid{K}: A$ is a canonical ideal of $A$.
\end{lemma}
\begin{proof}
By \cite[Lemma~2.11]{korell2018}, $\fracid K : A \in \RFI_R \cap \RFI_A$. Moreover, if $\fracid I \in \RFI_A$, then $\fracid I \in \RFI_R$ (see \cite[Lemma~2.12]{korell2018}). Therefore,
\[
\br{\fracid K : A} : \br{\br{\fracid K : A} : \fracid I} = \br{\fracid K : A} : \br{\fracid K : \fracid I} = \br{\fracid K : \br{\fracid K : \fracid I}} : A = \fracid I.
\qedhere
\]
\end{proof}


\begin{proposition}
\label{024}
Let $S$ be a good semigroup, and let $K$ be a canonical ideal of $S$. Then $E \subset K$ for every $E \in \GSI_S$ with $\gamma_E = \gamma_K$.
\end{proposition}
\begin{proof}
See \cite[Proof of Proposition~5.2.10]{j.commut.algebra.11.81}.
\end{proof}


\begin{lemma}
\label{040}
Let $S$ be a local good semigroup, and let $K$ be a canonical semigroup ideal of $S$. Then $\ol \Delta \br{\tau_K} \setminus \Delta \br{\tau_K} \subset K$.
\end{lemma}
\begin{proof}
By Theorem~\ref{018}.\ref{018b} we may assume that $K = K_S^0$, and hence $\tau_K = \tau_S$.

First note that $\ol \Delta \br{\tau_K} = \Delta \br{\tau_K}$ if $\vbr{I} = 1$. So we consider the case $\vbr{I} \geq 2$. Let $i \in I$, and let
\[
\alpha \in \ol{\Delta}_i^{K_S^0} \br{\tau_{K_S^0}} \setminus \Delta_i^{K_S^0} \br{\tau_{K_S^0}}.
\]
Then there is a $j \in I \setminus \cbr{i}$ such that
\begin{align*}
\alpha_i &= \br{\tau_S}_i, \\
\alpha_j &= \br{\tau_S}_j, \\
\alpha_k &\geq \br{\tau_S}_k \text{ for all } k \in I \setminus \cbr{i,j}
\end{align*}
(see Definition~\ref{039}). Hence,
\begin{align*}
\br{\tau_S - \alpha}_i &= 0, \\
\br{\tau_S - \alpha}_j &= 0, \\
\br{\tau_S - \alpha}_k &\leq 0 \text{ for all } k \in I \setminus \cbr{i,j}.
\end{align*}
By Remark~\ref{107} we have $\Delta_l^S \br{\tau_S - \alpha} = \emptyset$ for all $l \in I$ with $\br{\tau_S - \alpha}_l = 0$ as $S$ is local. Moreover, if $\br{\tau_S - \alpha}_l < 0$, then $\Delta_l^S \br{\tau_S - \alpha} = \emptyset$ since $S \geq \zero$. Thus, $\Delta^S \br{\tau_S - \alpha} = \emptyset$, and therefore $\alpha \in K_S^0$.
\end{proof}


\begin{lemma}
\label{021}
Let $S$ be a good semigroup, and let $E \in \GSI_S$. Then
\[
K_S^0 - E = \cbr{\alpha \in \ZZ^I \mid \Delta^E \br{\tau_S - \alpha} = \emptyset}.
\]
\end{lemma}
\begin{proof}
See \cite[Computation~3.3]{comm.algebra.25.2939}.
\end{proof}


\section{Gorenstein Endomorphism Rings}
\label{111}


In this section, we proof the following characterization of the Gorensteinness of the endomorphism ring $\End \br{\fracid I}$, where $\fracid I$ is a regular fractional ideal of an admissible ring.


\begin{theorem}
\label{001}
Let $R$ be an admissible ring, and let $\fracid{I}$ be a fractional ideal of $R$. Then the following are equivalent.
\begin{equivalence}
\item\label{001a}
$\fracid{I} : \fracid{I}$ is Gorenstein.
\item\label{001b}
$\Gamma_\fracid{I} - \Gamma_\fracid{I} = \Gamma_{\fracid{I}:\fracid{I}}$ and $\Gamma_\fracid{I} - \Gamma_\fracid{I}$ is a symmetric semigroup.
\item\label{001c}
$\fracid{I}$ is stable and $\Gamma_\fracid{I} - \Gamma_\fracid{I}$ is a symmetric semigroup.
\item\label{001d}
$\fracid I$ is stable and a canonical ideal for $\fracid I : \fracid I$.
\item\label{001e}
$\fracid I$ is stable and \emph{self-dual}, i.e.\ for any canonical ideal $\fracid K$ of $R$ there is an $x \in Q_R^\reg$ such that $\fracid I = x \br{\fracid K : \fracid I}$.
\end{equivalence}
\end{theorem}


The implication of stability from Gorensteinness relates Theorem~\ref{001} to the theory of almost Gorenstein rings. As shown in \cite{j.alg.379.355}, for a local one-dimensional Cohen--Macaulay ring $R$, the endomorphism ring $\fracid m : \fracid m$ of the maximal ideal $\fracid m$ is Gorenstein if and only if $R$ is almost Gorenstein of maximal embedding dimension. However, $R$ having maximal embedding dimension is equivalent to $\fracid m$ being stable, see \cite[Corollary~1.10]{amer.j.math.93.649}.


Similarly, considering again the maximal ideal of a local admissible ring of maximal embedding dimension, Theorem~\ref{001} is related to the result of a recent study of \emph{almost symmetric good semigroups} (see \cite[Definitions on page 228]{j.pure.app.alg.147.215}): In \cite[Corollary~4.2]{ric.mat.23} the authors find that a local admissible ring $R$ is almost symmetric if and only if $\Gamma_{\fracid m : \fracid m} = \Gamma_{\fracid m} - \Gamma_{\fracid m}$ and $\Gamma_{\fracid m}$ is a canonical ideal of $\Gamma_{\fracid m} - \Gamma_{\fracid m}$. In the case of maximal embedding dimension, this coincides again with the assertions of Theorem~\ref{001} applied to $\fracid I = \fracid m$.


From Section~\ref{109} we know that the difference $E - E$ is not necessarily good even if $E$ is the value semigroup ideal of a fractional ideal. In particular, differences are in general not compatible with quotients. In Section~\ref{110}, however, we learned that differences behave well in the context of dualizing, i.e.\ if canonical ideals are involved. Nonetheless, the following example shows that in Theorem~\ref{001}.\ref{001b}, only requiring symmetry of $\Gamma_\fracid{I} - \Gamma_\fracid{I}$ would not be sufficient.


\begin{figure}
\begin{center}
\begin{tikzpicture}[scale=0.5,inner sep=2.5]
\draw[->] (0,0) -- (15,0);
\foreach \i in {1,2,3,5,7,9} \draw (\i,0) node[shape=circle,draw,fill=white] {};

\foreach \i in {0,4,6,8,10,11,...,14} \draw (\i,0) node[shape=circle,draw,fill] {};

\draw (6.5,0.5) node[above] {$\Gamma_R$};
\end{tikzpicture}\\\medskip
\begin{tikzpicture}[scale=0.5,inner sep=2.5]
\draw[->] (0,0) -- (15,0);
\foreach \i in {1,3,5} \draw (\i,0) node[shape=circle,draw,fill=white] {};

\foreach \i in {0,2,4,6,7,...,14} \draw (\i,0) node[shape=circle,draw,fill] {};

\draw (6.5,0.5) node[above] {$\Gamma_{\fracid{m}_R} - \Gamma_{\fracid{m}_R}$};
\end{tikzpicture}\\\medskip
\begin{tikzpicture}[scale=0.5,inner sep=2.5]
\draw[->] (0,0) -- (15,0);
\foreach \i in {1,2,3,5} \draw (\i,0) node[shape=circle,draw,fill=white] {};

\foreach \i in {0,4,6,7,...,14} \draw (\i,0) node[shape=circle,draw,fill] {};

\draw (6.5,0.5) node[above] {$\Gamma_{\fracid{m}_R : \fracid{m}_R}$};
\end{tikzpicture}
\end{center}
\caption{The local algebroid curve $R = \CC \ssbr{t^4, t^6 + t^7, t^{11}}$ of Example~\ref{093} has semigroup of values $\Gamma_R = \abr{4,6,11,13}$ with maximal ideal $\Gamma_{\fracid{m}_R} = M_{\Gamma_R} = \Gamma_R \setminus \cbr{0}$. Then $\Gamma_{\fracid{m}_R} - \Gamma_{\fracid{m}_R} = \abr{2,7}$ and $\Gamma_{\fracid{m}_R : \fracid{m}_R} = \abr{4,6,7,9} \subsetneq \Gamma_{\fracid{m}_R} - \Gamma_{\fracid{m}_R}$.
\label{094}}
\end{figure}


\begin{example}
\label{093}
Consider the local algebroid curve $R = \CC \ssbr{t^4, t^6 + t^7, t^{11}}$ with maximal ideal $\fracid{m}_R = \abr{t^4, t^6 + t^7, t^{11}}$. Then $\Gamma_R = \abr{4,6,11,13}$ and $\Gamma_{\fracid{m}_R} - \Gamma_{\fracid{m}_R} = \abr{2,7}$ are good semigroups (cf.\ Remark~\ref{032}.\ref{032b}). Assume that $2 \in \Gamma_{\fracid{m}_R : \fracid{m}_R}$. Then there is by Proposition~\ref{075}.\ref{075c} a unique $a \in \br{\CC \ssbr{t}}^\ast$ such that $a t^2 \in \fracid{m}_R : \fracid{m}_R$. Therefore, there is a $b \in \CC \setminus \cbr{0}$ such that $a t^2 t^4 = b \br{t^6 + t^7}$. Hence, $a = b \br{1 + t}$. However, $b \br{t^2 + t^3} \br{t^6 + t^7} = b \br{t^8 + 2 t^9 + t^{10}}$ is not contained in $\fracid{m}_R$. This implies $2 \not\in \Gamma_{\fracid{m}_R : \fracid{m}_R}$, and hence $\Gamma_{\fracid{m}_R : \fracid{m}_R} \subsetneq \Gamma_{\fracid{m}_R} - \Gamma_{\fracid{m}_R}$. See Figure~\ref{094}.
\end{example}


In the following, we collect some preliminary results for the proof of Theorem~\ref{001}. With a view towards almost symmetric good semigroups, a similar result was recently found in \cite[Theorem~3.4]{ric.mat.23}.


\begin{lemma}
\label{002}
Let $S$ be a good semigroup, and let $E$ be a good semigroup ideal of $S$ with $\gamma_E = \gamma_S$. If $E - E$ is a symmetric good semigroup, then $E - E = K_S^0 - E$.
\end{lemma}
\begin{proof}
Since $\gamma_E = \gamma_S = \gamma_{K_S^0}$, we have $E \subset K_S^0$ by Proposition~\ref{024}, and hence $E - E \subset K_S^0 - E$ by Remark~\ref{003}.\ref{003c}.

Let now $\alpha \in \ZZ^I \setminus \br{E - E}$. Since $E - E$ is a symmetric good semigroup, we have
\[
E - E = K_{E-E}^0 = \cbr{\beta \in \ZZ^I \mid \Delta^{E-E} \br{\tau_{E-E} - \beta} = \emptyset}
\]
by Theorem~\ref{018}.\ref{018d}. Hence, there is a $\delta \in \Delta^{E-E} \br{\tau_{E-E} - \alpha}$. This implies
\[
\delta + \mu_E \in \Delta \br{\tau_{E-E} + \mu_E - \alpha} \cap \br{\br{E - E} + E} \subset \Delta \br{\tau_E - \alpha} \cap E = \Delta^E \br{\tau_S - \alpha}
\]
by Lemma~\ref{004}. Therefore, $\alpha \not\in K_S^0 - E$ by Lemma~\ref{021}. This yields $K_S^0 - E \subset E - E$. Hence, $E - E = K_S^0 - E$.
\end{proof}


With Lemma~\ref{002}, we can prove the combinatorial version of Theorem~\ref{001}. A similar result was recently found in \cite[Theorem~3.4]{ric.mat.23} with a view towards almost symmetric good semigroups: Considering the maximal ideal of a local good semigroup of maximal embedding dimension, \cite[Corollary~3.6]{ric.mat.23} coincides with the following.


\begin{proposition}
\label{012}
Let $S$ be a good semigroup, and let $E \in \GSI_S$. Then the following are equivalent:
\begin{equivalence}
\item\label{012a}
$E-E$ is a symmetric good semigroup.
\item\label{012b}
$E$ is stable and a canonical ideal for $E - E$.
\item\label{012c}
$E$ is stable and self-dual, i.e.\ for any canonical ideal $K$ of $S$ we have $\gamma_K - \gamma_E - \mu_E + E = K - E$.
\end{equivalence}
\end{proposition}
\begin{proof}\
\begin{description}[font=\normalfont]
\item[\ref{012a}$\implies$\ref{012b}]
First suppose that $\gamma_E = \gamma_S$. By Lemma~\ref{004} and Theorem~\ref{018}.\ref{018d} and \ref{018f} we have
\begin{align*}
K_{E-E}^0 &= \br{\gamma_{E-E} - \gamma_S} + \br{K_S^0 - \br{E-E}} \\
&= \br{\gamma_E - \mu_E - \gamma_S} + \br{K_S^0 - \br{E-E}} \\
&= - \mu_E + \br{K_S^0 - \br{E-E}}.
\end{align*}
Therefore, Theorem~\ref{018} and Lemma~\ref{002} yield
\[
E = K_S^0 - \br{K_S^0 - E} = K_S^0 - \br{E - E} = \mu_E + K_{E-E}^0 = \mu_E+ \br{E - E}.
\]
Thus, $E$ is stable.

Let now $\gamma_E$ be arbitrary. Then $\gamma_{\gamma_S - \gamma_E + E} = \gamma_S$, and $\br{\gamma_S - \gamma_E + E} - \br{\gamma_S - \gamma_E + E} = E - E$ is symmetric by Remark~\ref{003}.\ref{003b}. Thus, $\br{\gamma_S - \gamma_E + E}$ is stable. This implies
\begin{align*}
E - E &= \br{\gamma_S - \gamma_E + E} - \br{\gamma_S - \gamma_E + E} \\
&= - \mu_{\gamma_S - \gamma_E + E} + \br{\gamma_S - \gamma_E + E} \\
&= - \br{\gamma_S - \gamma_E + \mu_E} + \gamma_S - \gamma_E + E \\
&= - \mu_E + E.
\end{align*}
Therefore, $E$ is stable.

If $E$ is stable, then $E = \mu_E + \br{E - E}$. Since $E - E$ is symmetric, $E$ is a canonical ideal of $E - E$ by Theorem~\ref{018}.
\item[\ref{012b}$\implies$\ref{012c}]
Let $K$ be a canonical ideal of $S$. If $E$ is stable and a canonical ideal of $E-E$, then Theorem~\ref{018} implies
\begin{align*}
E &= \gamma_K - \gamma_E + \br{K - \br{E - E}} \\
&= \gamma_E - \gamma_K + \br{K - \br{-\mu_E + E}} \\
&= \gamma_E - \gamma_K + \mu_E +\br{K - E}.
\end{align*}
\item[\ref{012c}$\implies$\ref{012a}]
If $E$ is stable, then
\begin{align*}
E - E &= -\mu_E + E \\
&= -\mu_E + \gamma_E - \gamma_K + \mu_E +\br{K - E} \\
&= -\mu_E + \gamma_E - \gamma_K + \br{K - \br{-\mu_E + E}} \\
&= -\mu_E + \gamma_E - \gamma_K + \br{K - \br{E - E}}.
\end{align*}
Thus, $E - E$ is symmetric by Theorem~\ref{018}.
\qedhere
\end{description}
\end{proof}


\begin{lemma}
\label{013}
Let $R$ be an admissible ring, and let $\fracid{I} \in \RFI_R$. Then $\Gamma_\fracid{I} - \Gamma_\fracid{I} \subset K_{\Gamma_{\fracid{I} : \fracid{I}}}^0$.
\end{lemma}
\begin{proof}
By Remark~\ref{003}.\ref{003b} we may assume that $\gamma_{\Gamma_\fracid{I}} = \gamma_{\Gamma_R}$. Suppose that there is an $\alpha \in \br{\Gamma_\fracid{I} - \Gamma_\fracid{I}} \setminus K_{\Gamma_{\fracid{I} : \fracid{I}}}^0$. Then there is a
\[
\beta \in \Delta^{\Gamma_{\fracid{I}:\fracid{I}}} \br{\tau_{\Gamma_{\fracid{I} : \fracid{I}}} - \alpha}
\]
(see Definition~\ref{017}). So Lemma~\ref{004} yields
\[
\beta + \mu_{\Gamma_\fracid{I}} \in \Delta \br{\tau_{\Gamma_{\fracid{I} : \fracid{I}}} + \mu_{\Gamma_\fracid{I}} - \alpha} \cap \br{\br{\Gamma_\fracid{I} - \Gamma_\fracid{I}} + \Gamma_\fracid{I}} \subset \Delta^{\Gamma_\fracid{I}} \br{\tau_{\Gamma_\fracid{I}} - \alpha}.
\]
This implies by Lemma~\ref{021} $\alpha \not\in K_{\Gamma_R}^0 - \Gamma_\fracid{I}$, contradicting $\alpha \in \Gamma_\fracid{I} - \Gamma_\fracid{I} \subset K_{\Gamma_R}^0 - \Gamma_\fracid{I}$ (see Remark~\ref{003}.\ref{003c}, Definition~\ref{015}.\ref{015b}, and Proposition~\ref{024}).
\end{proof}


\begin{lemma}
\label{014}
Let $R$ be an admissible ring, and let $\fracid{I} \in \RFI_R$. If $\fracid{I} : \fracid{I}$ is Gorenstein, then $\Gamma_{\fracid{I}:\fracid{I}} = \Gamma_\fracid{I} - \Gamma_\fracid{I}$.
\end{lemma}
\begin{proof}
If  $\fracid{I} : \fracid{I}$ is Gorenstein, then Remark~\ref{022}.\ref{022c}, Theorem~\ref{018}.\ref{018d}, Proposition~\ref{020}, and Lemma~\ref{013} yield
\[
\Gamma_{\fracid{I}:\fracid{I}} \subset \Gamma_\fracid{I} - \Gamma_\fracid{I} \subset K_{\Gamma_{\fracid{I}:\fracid{I}}}^0 = \Gamma_{\fracid{I}:\fracid{I}}.
\qedhere
\]
\end{proof}


Collecting the above results, we obtain the proof of Theorem~\ref{001}.


\begin{proof}[Proof of Theorem~\ref{001}]\
\begin{description}[font=\normalfont]
\item[\ref{001a}$\implies$\ref{001b}]
If $\fracid{I}:\fracid{I}$ is Gorenstein, then $\Gamma_\fracid{I} - \Gamma_\fracid{I} = \Gamma_{\fracid{I}:\fracid{I}}$ by Lemma~\ref{014}. Hence, $\Gamma_\fracid{I} - \Gamma_\fracid{I}$ is symmetric by Proposition~\ref{020}.
\item[\ref{001b}$\implies$\ref{001c}]
If $\Gamma_\fracid{I} - \Gamma_\fracid{I}$ is symmetric, then $\Gamma_\fracid{I}$ is stable by Proposition~\ref{012}. So if $\Gamma_\fracid{I} - \Gamma_\fracid{I} = \Gamma_{\fracid{I}:\fracid{I}}$, then Proposition~\ref{011} implies that $\fracid{I}$ is stable.
\item[\ref{001c}$\implies$\ref{001d}]
If $\fracid I$ is stable, then $\Gamma_{\fracid I : \fracid I} = \Gamma_{\fracid I} - \Gamma_{\fracid I}$ is symmetric by Proposition~\ref{011}. By Proposition~\ref{020}, $\fracid I : \fracid I$ is therefore Gorenstein. With the stability of $\fracid I$, it follows from Proposition~\ref{034} that $\fracid I$ is a canonical ideal of $\fracid I : \fracid I$.
\item[\ref{001d}$\implies$\ref{001e}]
If $\fracid I$ is stable, then there is $y \in Q_R^\reg$ such that $\fracid I = y \br{\fracid I : \fracid I}$. If $\fracid I$ is a canonical ideal of $\fracid I : \fracid I$, then Proposition~\ref{034} and Lemma~\ref{035} yield
\[
\fracid I = z \br{\fracid K : \br{\fracid I : \fracid I}} = yz \br{\fracid K : \fracid I}
\]
for some $z \in Q_R^\reg$.
\item[\ref{001e}$\implies$\ref{001a}]
If $\fracid I$ is stable, there is $y \in Q_R^\reg$ such that $y \fracid I = \br{\fracid I : \fracid I}$. Thus,
\[
\fracid I : \fracid I = y \fracid I = y x \br{\fracid K : \fracid I} = x \br{\fracid K : \br{\fracid I : \fracid I}}
\]
is Gorenstein by Proposition~\ref{034} and Lemma~\ref{035}.
\qedhere
\end{description}
\end{proof}


\section{Symmetric Semigroups}
\label{112}


The aim for the rest of the paper is to characterize in Theorem~\ref{026} the Gorenstein local admissible rings for which the endomorphism ring of the maximal ideal is Gorenstein again. To achieve this, we first prove the combinatorial version in Theorem~\ref{023}.


\subsection{Differences in Symmetric Semigroups}


By Proposition~\ref{020}, an admissible ring is Gorenstein if and only if its value semigroup is symmetric. We therefore begin by studying differences in symmetric semigroups more closely.


\begin{lemma}
\label{038}
Let $S$ be a good semigroup, and let $\alpha \in \ZZ^I$ with $\alpha \leq \gamma_S$. Then $S^\alpha - S^\alpha = S - S^\alpha$.
\end{lemma}
\begin{proof}
Definition~\ref{052}.\ref{052c} and Remark~\ref{053}.\ref{053b} yield $\mu_{S - S^\alpha} \geq \gamma_S - \gamma_{S^\alpha} = \gamma_S - \gamma_S = 0$. This implies with Remark~\ref{053}.\ref{053a}
\[
\br{S - S^\alpha} + S^\alpha \subset \br{\ZZ^I}^\alpha \cap S = S^\alpha,
\]
and hence $S - S^\alpha \subset S^\alpha - S^\alpha$.

The claim follows since $S^\alpha - S^\alpha \subset S - S^\alpha$ by Remark~\ref{003}.\ref{003c}.
\end{proof}


\begin{lemma}
\label{037}
Let $S$ be a good symmetric semigroup, and let $\alpha \in \ZZ^I$ with $\alpha \leq \gamma_S$. Then $S^\alpha - S^\alpha$ is a good semigroup with $S \subset S^\alpha - S^\alpha \subset \ol S$. In particular, if $S = \ol S$, then $S^\alpha - S^\alpha = S$.
\end{lemma}
\begin{proof}
A straightforward argument shows that $S^\alpha - S^\alpha$ is a submonoid of $\ZZ^I$ and $S \subset S^\alpha - S^\alpha \subset \ol{S^\alpha - S^\alpha} = \ol S$ (see \cite[Proposition~4.25]{korell2018}). Moreover, $S^\alpha - S^\alpha$ satisfies property~\ref{E0} by Remarks~\ref{032}.\ref{032a}, \ref{055}, and \ref{025}.\ref{025b}.

Since $S$ is symmetric, i.e.\ $S$ is a canonical ideal of itself, Remark~\ref{025}.\ref{025b}, Theorem~\ref{018}.\ref{018e}, and Lemma~\ref{038} yield
\[
S^\alpha - S^\alpha = S - S^\alpha \in \GSI_S.
\]
Thus, $S^\alpha - S^\alpha$ satisfies properties~\ref{E1} and \ref{E2}.
\end{proof}


The following Lemma was proved for numerical semigroups in \cite[Satz~1.9.(e)]{HK71}. Since $S - M_S = S - S$ (see \cite[Proposition~4.38]{korell2018}), the general case follows from \cite[Lemma~3.5]{j.pure.app.alg.147.215}. Here we provide a different proof (cf.~\cite[Lemma~5.44]{korell2018}).


\begin{lemma}
\label{042}
Let $S$ be a local symmetric semigroup with maximal ideal $M_S$. If $S \ne \ol S$, then $M_S - M_S = S \cup \Delta \br{\tau_S}$.
\end{lemma}
\begin{proof}
Lemmas~\ref{004} and \ref{037} imply $S \cup \Delta^S \br{\tau_S} \subset M_S - M_S$ since $\mu_{M_S} \geq \one$ by Remark~\ref{107}.

Now suppose that there is an $\alpha \in \br{M_S - M_S} \setminus \br{S \cup \Delta \br{\tau_S}}$. Since $S$ is symmetric, there is a $\beta \in \Delta^S \br{\tau_S - \alpha}$. By assumption and Lemma~\ref{040} we have $\alpha \not\in \ol{\Delta}^S \br{\tau_S}$. Therefore, $\tau_S - \alpha > \zero$, and hence $\beta \in \Delta^{M_S} \br{\tau_S - \alpha}$. This implies
\[
\alpha + \beta \in \br{\alpha + \Delta \br{\tau_S - \alpha}} \cap \br{\br{M_S - M_S} + M_S} \subset \Delta^{M_S} \br{\tau_S} \subset \Delta^S \br{\tau_S} = \emptyset,
\]
where the last equality follows from properties~\ref{E1} and \ref{E2} (see \cite[Lemma~4.1.10]{j.commut.algebra.11.81}).
\end{proof}


\subsection{Local Symmetric Semigroups}


Let $S$ be a local good semigroup with maximal ideal $M_S$. We can now fully classify when $S$ and $M_S - M_S$ are symmetric.


\begin{theorem}
\label{023}
Let $S$ be a local good semigroup with maximal ideal $M_S$. The following are equivalent.
\begin{equivalence}
\item\label{023a}
Every good semigroup $S'$ with $S \subset S' \subset \ol S$ is symmetric.
\item\label{023b}
$S$ and $M_S - M_S$ are symmetric good semigroups.
\item\label{023c}
There is an $n \in 2 \NN$ such that
\[
S = \abr{2, n+1},
\]
or there is an $n \in 2 \NN + 1$ such that
\[
S = \abr{\one} \cup \br{\br{\frac{n+1}{2}}_{i \in I} + \NN^{I}},
\]
where $\vbr{I} = 2$.
\end{equivalence}
\end{theorem}


Note that the equivalence of statements \ref{023a} and \ref{023c} of Theorem~\ref{023} was shown in \cite[Lemma 1.15]{HK71} for numerical semigroups.


\begin{lemma}
\label{059}
Let $S$ be a numerical symmetric semigroup. Then $\gamma_S = 2 \vbr{\ol S \setminus S}$.
\end{lemma}
\begin{proof}
See \cite[Satz~1.9.(c)]{HK71}.
\end{proof}


\begin{proposition}
\label{057}
Let $S$ be a numerical symmetric semigroup. If $M_S - M_S$ is symmetric (see Remark~\ref{095}), then $S = \abr{2, n+1}$ for some $n \in 2 \NN$.
\end{proposition}
\begin{proof}
Suppose that $S \ne \ol S$ (see Lemma~\ref{037}), and that $M_S - M_S$ is symmetric. Since $M_S - M_S = S \cup \cbr{\tau_S}$ by Lemma~\ref{042}, Lemma~\ref{059} yields
\[
\gamma_{M_S - M_S} = 2\, \vbr{\ol S \setminus \br{M_S - M_S}} = 2\, \br{\vbr{\ol S \setminus S} - 1} = \gamma_S - 2.
\]
Thus, we have $2 = \gamma_S - \gamma_{M_S - M_S} = \mu_{M_S} \in M_S \subset S$ by Lemma~\ref{004}, and by Remark~\ref{107} the only smaller element in $S$ is $0$. This implies $S = \abr{2, \gamma_S+1}$, where $\gamma_S = 2 \vbr{\ol S \setminus S}$ is even.
\end{proof}


\begin{lemma}
\label{060}
Let $S$ be a local symmetric semigroup. If $M_S - M_S$ is symmetric, then $\alpha - \mu_{M_S} \in S$ for every $\alpha \in M_S \setminus C_S$.
\end{lemma}
\begin{proof}
By Proposition~\ref{012} and Lemma~\ref{042} we have
\[
\alpha \in M_S = \mu_{M_S} + \br{M_S - M_S} = \mu_{M_S} + \br{S \cup \Delta \br{\tau_S}} = \br{\mu_{M_S} + S} \cup \Delta \br{\tau_S + \mu_{M_S}}.
\]
Assume that $\alpha \in \Delta \br{\tau_S + \mu_{M_S}}$. Then $\alpha \geq \gamma_S$ since $\mu_{M_S} \geq \one$ by Remark~\ref{107}. But this is a contradiction to the choice of $\alpha$. Hence, $\alpha - \mu_{M_S} \in S$.
\end{proof}


\begin{lemma}
\label{061}
Let $S$ be a local symmetric semigroup. If $M_S - M_S$ is symmetric, then for every $\alpha \in M_S \setminus C_S$ there is an $n \in \NN$ such that $\alpha = n \mu_{M_S}$. In particular, $S = \abr{\mu_{M_S}} \cup C_S$.
\end{lemma}
\begin{proof}
Let $\alpha \in M_S \setminus C_S$. Since $S = \cbr{\zero} \cup M_S$ by Remark~\ref{107}, and since $C_S = \gamma_S + \ol S$ by Remark~\ref{053}.\ref{053b}, repeatedly applying Lemma~\ref{060} yields $\alpha - m \mu_{M_S} \in S$ for all $m \in \NN$ satisfying $\br{m - 1} \mu_{M_S} < \alpha$. Since $\alpha$ is finite, there is $n = \max \cbr{m \in \NN \mid \br{m - 1} \mu_{M_S} < \alpha}$. Then $\alpha - n \mu_{M_S} \in S$ implies $\alpha - n \mu_{M_S} \geq \zero$. Hence, $n \mu_{M_S} \not< \alpha$ yields $\alpha = n \mu_{M_S}$. The particular claim follows since $S = \cbr{\zero} \cup M_S$ by Remark~\ref{107}.
\end{proof}


\begin{lemma}
\label{062}
Let $S$ be a local symmetric semigroup. If $M_S - M_S$ is symmetric, then $\vbr{I} \leq 2$.
\end{lemma}
\begin{proof}
If $\vbr{I} \geq 3$, then $\ol \Delta \br{\tau_S} \setminus \Delta \br{\tau_S} \subset S$ contradicts Lemma~\ref{061}.
\end{proof}


\begin{proposition}
\label{063}
Let $S$ be a local symmetric semigroup with $\vbr{I} = 2$. If $M_S - M_S$ is symmetric, then $\mu_{M_S} = \one$. Moreover, there is an $n \in 2 \NN + 1$ such that
\[
S = \abr{\one} \cup \br{\br{\frac{n+1}{2}}_{i \in I} + \NN^I}.
\]
\end{proposition}
\begin{proof}
We may suppose that $I = \cbr{1,2}$. Since $\mu_{M_S} \geq \one$ by Remark~\ref{107}, Lemma~\ref{004} yields $\gamma_{M_S - M_S} = \gamma_{M_S} - \mu_{M_S} \leq \gamma_S - \one = \tau_S$ (note that $S \ne \ol S$ by Remark~\ref{107}). Assume that $\gamma_{M_S - M_S} < \tau_S$. Then $\ol \Delta_i \br{\tau_S - \ee_i} \subset M_S - M_S$ for an $i \in I$. However, $\ol \Delta_i \br{\tau_S - \ee_i} \not\subset S$ by Lemma~\ref{061}, and $\ol \Delta_i \br{\tau_S - \ee_i} \cap \Delta \br{\tau_S} = \emptyset$ by Definition~\ref{039}. So with Lemma~\ref{042} we obtain the contradiction $\ol \Delta_i \br{\tau_S - \ee_i} \not\subset S \cup \Delta \br{\tau_S} = M_S - M_S$. Therefore, $\gamma_{M_S - M_S} = \tau_S$, and Lemma~\ref{004} yields $\mu_{M_S} = \gamma_S - \gamma_{M_S - M_S} = \gamma_S - \tau_S = \one$. Thus, we have $S = \abr{\one} \cup C_S$ by Lemma~\ref{061}.

Without loss of generality we may assume that $\br{\gamma_S}_1 \leq \br{\gamma_S}_2$. Let $\alpha \in \NN$ with $\alpha \geq \br{\gamma_S}_1$. Then Lemma~\ref{064} implies $\br{\alpha, \alpha} + \NN \ee_1 \subset S$. Let now $m \in \NN$. Then $\br{\alpha + m, \alpha + m} \in S$. Moreover, $\alpha \geq \br{\gamma_S}_1$ and $\max \cbr{\alpha + m, \br{\gamma_S}_2} \geq \br{\gamma_S}_2$ imply $\br{\alpha, \max \cbr{\alpha + m, \br{\gamma_S}_2}} \in S$. Since $S$ satisfies property~\ref{E1}, this yields
\[
\br{\alpha, \alpha + m} = \inf \br{\br{\alpha + m, \alpha + m}, \br{\alpha, \max \cbr{\alpha + m, \br{\gamma_S}_2}}} \in S,
\]
and hence $\br{\alpha,\alpha} + \NN \ee_2 \subset S$. Since also $\br{\alpha, \alpha} + \NN \ee_1 \subset S$ by Lemma~\ref{064}, we obtain $\br{\br{\gamma_S}_1, \br{\gamma_S}_1} + \ol S \subset S$, and hence $\gamma_S = \br{\br{\gamma_S}_1, \br{\gamma_S}_1}$.

Therefore, setting $n = 2 \br{\gamma_S}_1 - 1 \in 2 \NN + 1$ we obtain
\[
S = \abr{\one} \cup C_S = \abr{\one} \cup \br{\br{\frac{n+1}{2},\frac{n+1}{2}} + \NN^I}
\]
(see Remark~\ref{053}.\ref{053b}).
\end{proof}


\begin{proof}[Proof of Theorem~\ref{023}]\
\begin{description}[font=\normalfont]
\item[\ref{023a}$\implies$\ref{023b}]
Suppose that any good semigroup $S'$ with $S \subset S' \subset \ol S$ is symmetric. Then, in particular, $S$ is symmetric. Therefore, $M_S - M_S = S - M_S$ (see Remark~\ref{107} and Lemma~\ref{038}) is a good semigroup by Lemma~\ref{037}. Since $S \subset M_S - M_S \subset \ol S$ by Lemma~\ref{037}, it is symmetric by assumption.
\item[\ref{023b}$\implies$\ref{023c}]
Suppose that $S$ and $M_S - M_S$ are symmetric semigroups. Then $\vbr{I} \leq 2$ by Lemma~\ref{062}. The statement follows from Propositions~\ref{057} and \ref{063}.
\item[\ref{023c}$\implies$\ref{023a}]
If $S = \abr{2, n+1}$ for some $n \in 2 \NN$, the proof is in \cite[Lemma~1.15]{HK71}.

So suppose that $S = \abr{\one} \cup \br{\br{\frac{n+1}{2}}_{i \in I} + \NN^{I}}$ for some $n \in 2 \NN + 1$. Let $S'$ be a good semigroup with $S \subset S' \subset \ol S$. Since $S'$ has to satisfy properties \ref{E1} and \ref{E2}, it is easy to see that there is an $m \in 2 \NN + 1$ such that $S' = \abr{\one} \cup \br{\br{\frac{m+1}{2}}_{i \in I} + \ol S}$ (see Figure~\ref{065}).
\qedhere
\end{description}
\end{proof}


\begin{figure}
\begin{center}
\begin{tikzpicture}[scale=0.5,inner sep=2.5]
\draw[->] (0,0) -- (17,0);
\foreach \i in {1,3,5,7,9,11} \draw (\i,0) node[shape=circle,draw,fill=white] {};

\foreach \i in {0,2,4,6,8,10,12,13,14,15,16} \draw (\i,0) node[shape=circle,draw,fill] {};

\draw (6.5,0.5) node[above] {$S$};
\end{tikzpicture}\\\medskip
\begin{tikzpicture}[scale=0.5,inner sep=2.5]
\draw[->] (0,0) -- (17,0);
\foreach \i in {1,3,5,7} \draw (\i,0) node[shape=circle,draw,fill=white] {};

\foreach \i in {0,2,4,6,8,9,10,11,12,13,14,15,16} \draw (\i,0) node[shape=circle,draw,fill] {};

\draw (6.5,0.5) node[above] {$S'$};
\end{tikzpicture}
\end{center}
\caption{The symmetric numerical semigroup $S = \abr{2, 6 +1}$. Every good semigroup $S'$ with $S \subset S' \subset \ol S$ is of the same form, e.g.\ $S' = \abr{2, 4 + 1}$.}
\label{066}
\end{figure}


\begin{figure}
\begin{center}
\begin{tikzpicture}[scale=0.5,inner sep=2.5]
\draw[->] (0,0) -- (11,0);
\draw[->] (0,0) -- (0,9);
\foreach \i in {0,...,10} \foreach \j in {0,...,8} \draw (\i,\j) node[shape=circle,draw,fill=white] {};

\foreach \i in {0,1,2,3,4} \draw (\i,\i) node[shape=circle,draw,fill] {};
\foreach \i in {5,6,7,8,9,10} \foreach \j in {5,6,7,8} \draw (\i,\j) node[shape=circle,draw,fill] {};

\draw (6.5,8.5) node[above] {$S$};
\end{tikzpicture}
\qquad
\begin{tikzpicture}[scale=0.5,inner sep=2.5]
\draw[->] (0,0) -- (11,0);
\draw[->] (0,0) -- (0,9);
\foreach \i in {0,...,10} \foreach \j in {0,...,8} \draw (\i,\j) node[shape=circle,draw,fill=white] {};

\foreach \i in {0,1,2} \draw (\i,\i) node[shape=circle,draw,fill] {};
\foreach \i in {3,4,5,6,7,8,9,10} \foreach \j in {3,4,5,6,7,8} \draw (\i,\j) node[shape=circle,draw,fill] {};

\draw (6.5,8.5) node[above] {$S'$};
\end{tikzpicture}
\end{center}
\caption{Every good semigroup containing $S = \abr{\br{1,1}} \cup \br{\br{\frac{9+1}{2},\frac{9+1}{2}} + \NN^2}$ and any point of $\ol \Delta \br{\br{3,3}} \setminus \cbr{\br{3,3}}$ must also contain the good semigroup $S' = \abr{\br{1,1}} \cup \br{\br{\frac{5+1}{2},\frac{5+1}{2}} + \NN^2}$ in order to satisfy properties~\ref{E1} and \ref{E2}.}
\label{065}
\end{figure}


\section{Gorenstein Algebroid Curves}
\label{113}


We conclude the paper with the following characterization of Gorenstein admissible rings which only have Gorenstein integral extensions in their ring of fractions. The equivalence of statements \ref{026a}, \ref{026c}, and \ref{026c} in Theorem~\ref{026} was proved for irreducible admissible rings already in \cite[Proposition 2.22]{HK71}.


\begin{theorem}
\label{026}
Let $R$ be a local admissible ring with maximal ideal $\fracid{m}_R$. The following are equivalent.
\begin{equivalence}
\item\label{026a}
Every integral extension of $R$ in its total ring of fractions $Q_R$ is Gorenstein.
\item\label{026b}
$R$ and $\fracid{m}_R : \fracid{m}_R$ are Gorenstein.
\item\label{026e}
$R$ is Gorenstein and has maximal embedding dimension.
\item\label{026f}
$\Gamma_R$ and $M_{\Gamma_R} - M_{\Gamma_R}$ are symmetric.
\item\label{026c}
There is an $n \in 2 \NN$ such that
\[
\Gamma_R = \abr{2, n+1},
\]
or there is an $n \in 2 \NN + 1$ such that
\[
\Gamma_R = \abr{\one} \cup \br{\br{\frac{n+1}{2}}_{V \in \VR_R} + \NN^{\VR_R}},
\]
where $\vbr{\VR_R} = 2$.
\end{equivalence}

If $R$ contains a field, and if $\KK = R / \fracid{m}_R$ is algebraically closed, then the above statements are equivalent to the following.
\begin{equivalence}[resume]
\item\label{026d}
There is a surjective $\KK$-algebra homomorphism
\[
\phi \colon \KK \ssbr{x,y} \to \wh R
\]
with
\[
\ker \phi = \abr{x^{n+1} - y^2}
\]
for some $n \in \NN$ (note that this is the same $n$ as in \ref{026c}).
\end{equivalence}
\end{theorem}


\begin{proof}\
\begin{description}[font=\normalfont]
\item[\ref{026a}$\implies$\ref{026b}]
This follows trivially as $R$ and $\fracid{m}_R : \fracid{m}_R$ are integral extensions of $R$.
\item[\ref{026b}$\implies$\ref{026f}]
If $R$ is Gorenstein, then $\Gamma_R$ is symmetric by Proposition~\ref{020}. Hence, Remark~\ref{102}.\ref{102c}, Proposition~\ref{086}, Theorems~\ref{016}.\ref{016b} and \ref{105}.\ref{105c}, and Lemma~\ref{038} yield
\[
\Gamma_{\fracid{m}_R : \fracid{m}_R} = \Gamma_{R : \fracid{m}_R} = \Gamma_R - \Gamma_{\fracid{m}_R} = \Gamma_R - M_{\Gamma_R} = M_{\Gamma_R} - M_{\Gamma_R}.
\]
Therefore, $\Gamma_R$ and $M_{\Gamma_R} - M_{\Gamma_R}$ are symmetric by Propositions~\ref{029}.\ref{029a} and \ref{020}.
\item[\ref{026f}$\implies$\ref{026c}]
This follows from Theorem~\ref{023}.
\item[\ref{026c}$\implies$\ref{026a}]
Suppose that
\[
\Gamma_R = \abr{2, n+1},
\]
for some $n \in 2 \NN$ or
\[
\Gamma_R = \abr{\one} \cup \br{\br{\frac{n+1}{2}}_{V \in \VR_R} + \NN^{\VR_R}}
\]
for some $n \in 2 \NN + 1$. Let $A$ be an integral extension of $R$ in $Q_R$. Then $A$ is an admissible ring by Proposition~\ref{029}.\ref{029a}, and $A \in \RFI_R$ with $\Gamma_R \subset \Gamma_A \subset \ol{\Gamma_R}$ (see Remark~\ref{022}.\ref{022a}). Thus, $\Gamma_A$ is symmetric by Theorem~\ref{023} (also see Proposition~\ref{029}.\ref{029a}), and hence $A$ is Gorenstein by Proposition~\ref{020}.
\item[\ref{026b}$\implies$\ref{026e}]
If $\fracid{m}_R : \fracid{m}_R$ is Gorenstein, then $R$ has maximal embedding dimension by \cite[Theorem~5.1]{j.alg.379.355}.
\item[\ref{026e}$\implies$\ref{026b}]
If $R$ is Gorenstein, then it is almost Gorenstein, see \cite[p.~364]{j.alg.379.355}. Since $R$ has maximal embedding dimension, also $\fracid{m}_R : \fracid{m}_R$ is Gorenstein by \cite[Theorem~5.1]{j.alg.379.355}.
\end{description}
\medskip
Now suppose that $R$ contains a field, and that $\KK = R / \fracid{m}_R$ is algebraically closed. By \cite[Corollary of Theorem~56]{matsumura1980} $\wh R$ is local with residue field $\KK$. Since $\KK$ is algebraically closed, it is infinite. In particular, this means $\vbr{\KK} \geq \vbr{\Min \br{\wh R}}$. By \cite[Corollary~2.1.8]{bruns1998} we have $\dim \wh R = \dim R$. In particular, $\wh R$ is equidimensional. Moreover, $\wh R$ is Noetherian by \cite[Theorem~10.26]{atiyah1969}, and it is reduced since $R$ is admissible. Finally, $\wh R$ is a $\KK$-algebra by the Cohen Structure Theorem (see \cite[Theorem~7.7]{eisenbud1995}). Thus, $\wh R$ is an algebroid curve over $R$. Since $\Gamma_R = \Gamma_{\wh R}$ by \cite[Theorem~3.3.5]{j.commut.algebra.11.81}, we may assume that $R$ is an algebroid curve over $\KK$.
\begin{description}[font=\normalfont]
\item[\ref{026c}$\implies$\ref{026d}]
First suppose that there is an $n \in 2 \NN$ such that $\Gamma_R = \abr{2, n+1}$. We identify $R \subset Q_R = \KK \ssbr{t} \sbr{t^{-1}}$ (see Proposition~\ref{078}). Since $2 \in \Gamma_R$, there is $x \in R$ with $\nu_R \br{x} = 2$. Then by Proposition~\ref{075}.\ref{075b} there is an $a \in \br{\KK \ssbr{t}}^\ast$ such that $x = at^2$. Since $\KK$ is algebraically closed, there is a $b \in \br{\KK \ssbr{t}}^\ast$ such that $b^2 = a$, and hence $x = \br{b t}^2$. By Propositions~\ref{075}.\ref{075a} and \ref{075b} and \ref{078} $b t$ is a uniformizing parameter for $\KK \ssbr{t}$. Therefore, we may by Proposition~\ref{078} assume that $x = t^2$.

Since $\fracid{C}_R = t^{\gamma_{\Gamma_R}} \KK\ssbr{t}$ by Proposition~\ref{082}, we have $t^{\gamma_{\Gamma_R} + 1} \in R$. Then $R' = \KK \ssbr{t^2, t^{\gamma_{\Gamma_R} + 1}}$ is an algebroid curve over $\KK$ with $\ol{R'} = \ol R = \KK \ssbr{t}$. Moreover, we have $\Gamma_R = \Gamma_{R'}$, and hence $\fracid{C}_R = \fracid{C}_{R'} \subset R' \subset R \subset Q_{R'}$ by Proposition~\ref{082}. Therefore, $R = \KK \ssbr{t^2, t^{\gamma_{\Gamma_R} + 1}}$ by Lemma~\ref{081}.\ref{081b}. Then the $\KK$-algebra homomorphism
\begin{align*}
\phi \colon \KK \ssbr{x,y} &\to R, \\
x &\mapsto t^2, \\
y &\mapsto t^{n+1},
\end{align*}
has kernel $\abr{x^{n+1} - y^2}$ (see \cite[Algorithm~3.6 and Corollary~3.8]{greuel2007} or \cite[Proposition~5.54]{korell2018}).

\medskip

Now suppose that there is an $n \in 2 \NN + 1$ such that 
\[
\Gamma_R = \abr{\br{1,1}} \cup \br{\br{\frac{n+1}{2},\frac{n+1}{2}} + \NN^{\VR_R}}
\]
We identify $R \subset Q_R = \KK \ssbr{t_1} \sbr{t_1^{-1}} \times \KK \ssbr{t_2} \sbr{t_2^{-1}}$ (see Proposition~\ref{078}). Since $\br{1,1} \in \Gamma_R$, we may by Propositions~\ref{075} and \ref{078} choose the parameters $t_1$ and $t_2$ such that $\br{t_1, t_2} \in R$.

Set $\gamma = \frac{n+1}{2}$. Then $\gamma_{\Gamma_R} = \br{\gamma,\gamma}$. So by Proposition~\ref{082} we have $\br{t_1^\gamma, -t_2^\gamma} \in \fracid{C}_R \subset R$. Then $R'' = \KK \ssbr{\br{t_1,t_2}, \br{t_1^\gamma,-t_2^\gamma}}$ is an algebroid curve over $\KK$ with $\ol{R''} = \ol R = \KK \ssbr{t_1} \times \KK \ssbr{t_2}$. Moreover, we have $\Gamma_R = \Gamma_{R''}$, and hence $\fracid{C}_R = \fracid{C}_{R''} \subset R'' \subset R \subset Q_{R''}$ by Proposition~\ref{082}. Therefore, $R = \KK \ssbr{\br{t_1,t_2}, \br{t_1^\gamma,-t_2^\gamma}}$ by Lemma~\ref{081}.\ref{081b}. Then the $\KK$-algebra homomorphism
\begin{align*}
\psi \colon \KK \ssbr{x,y} &\to R, \\
x &\mapsto \br{t_1, t_2}, \\
y &\mapsto \br{t_1^{\frac{n+1}{2}}, -t_2^{\frac{n+1}{2}}},
\end{align*}
has kernel
\[
\abr{x^{\frac{n+1}{2}} - y} \cap \abr{x^{\frac{n+1}{2}} + y} = \abr{x^{n+1} - y^2}
\]
(see \cite[Algorithm~3.6 and Corollary~3.8]{greuel2007} or \cite[Proposition~5.54]{korell2018}).
\item[\ref{026d}$\implies$\ref{026c}]
This follows by computing the semigroup of values of $\KK \ssbr{x,y} / \abr{x^{n+1} - y^2}$ (see \cite{j.lond.math.soc.60.420}).
\qedhere
\end{description}
\end{proof}


\bibliographystyle{unsrt}
\bibliography{ges}


\end{document}